\documentclass[a4paper]{amsart} 
\usepackage{graphicx,epsfig,amscd,graphics,xypic}
\usepackage{amssymb}
\usepackage{amscd}
\usepackage{times}
\usepackage{enumitem}
\usepackage{mathrsfs}
\usepackage[latin1]{inputenc}
\usepackage{psfrag}
\usepackage[active]{srcltx}
\usepackage{color}
\usepackage{wrapfig}

\numberwithin{equation}{section}

\newcommand{\dd}{\delta}

\newtheorem{thm}{Theorem}[section]
  
\newtheorem{coro}[thm]{Corollary}  
\newtheorem*{thm*}{Theorem}   
\newtheorem*{lemma*}{Lemma}

\newtheorem{prop}[thm]{Proposition}
           
\newtheorem{lemma}[thm]{Lemma}
\newtheorem{rem}[thm]{Remark}
\newtheorem{ex}[thm]{Exemple}
\newtheorem*{ex*}{Exemple \ref{exJ2} (continued)}

\newcommand{\Bl}{\operatorname{Bl}}

\newcommand{\coes}{\c c\~oes }

\renewcommand{\O}{{\mathcal O}}

\newcommand{\I}{{\mathcal I}}

\newcommand{\p}{{\mathbb P}}

\newcommand{\rk}{\operatorname{rk}}

\newcommand{\map}{\dasharrow}

\newcommand{\im}{\operatorname{Im}}

\renewcommand{\th}{\operatorname{th}}

\newcommand{\Str}{\operatorname{Str}}
\newcommand{\Conf}{\operatorname{Conf}}
\def\S{Section~}

\def\og{\leavevmode\raise.3ex\hbox{$\scriptscriptstyle\langle\!\langle$~}}
\def\fg{\leavevmode\raise.3ex\hbox{~$\!\scriptscriptstyle\,\rangle\!\rangle$}}

\def\ro[#1]{{\textcolor{black}{#1}}}
\def\ri[#1]{{\textcolor{green}{#1}}}

\begin{document}

\title[Varieties $n$--covered by curves of degree $\dd$]{On  projective varieties $n$--covered by  curves of degree $\dd$}  

\author[Luc Pirio]{Luc Pirio} 
\address{Luc PIRIO, IRMAR, UMR 6625 du CNRS, Universit\' e Rennes1, Campus de beaulieu, 
35000 Rennes, France}
\email{luc.pirio@univ-rennes1.fr}
\author[Francesco Russo]{Francesco Russo} 
\address{Francesco RUSSO, Dipartimento di Matematica e Informatica Universit\` a degli Studi di Catania 95125  Catania, Italy}
\email{frusso@dmi.unict.it}

\maketitle


\section*{Introduction}

The  theory of  rationally connected varieties
is quite recent and  was formalized in  \cite{campana}, \cite{KMM} and \cite{kollar}, although  these varieties
were intensively studied from different points of view by classical algebraic geometers, see for example \cite{darboux}, \cite{segre}, \cite{scorza}, \cite{bompiani}, \cite{bol}, \cite{lane}.\smallskip 

An important  result in this theory, see \cite[Theorem IV.3.9]{kollar},  asserts that through  $n$ general  points of a smooth rationally connected complex variety $X$ there passes  an irreducible rational  curve, which can be taken also to be  smooth as soon as $\dim(X)\geq 3$. From this one deduces that  for $\dim(X)\geq 3$ a fixed smooth curve of arbitrary genus can be embedded into $X$ in such a way that it passes through  $n$ arbitrary fixed general points.
 When a (rationally connected) variety  $X$ is embedded in some projective space $\p^N$ (or more generally when a polarization or an arbitrary Cartier divisor is fixed on $X$), one can consider the property of being  generically $n$-(rationally) connected by (rational) curves of a fixed degree $\dd$.

This stronger condition depends on the embedding, on the number $n\geq 2$, on the degree $\delta\geq 1$ and natural constraints  for the existence of such varieties immediately appear.
\smallskip

In this paper we shall study complex irreducible projective varieties   $X={X}(r+1, n,\delta)\subset\p^N$ of dimension $\dim(X)=r+1$ such that through $n\geq 2$ general points there passes an irreducible curve $C$ of degree $\dd\geq 1$ or more generally pairs $(X,D)$ with  $D$ a Cartier divisor
on a proper irreducible complex variety $X$ which is  $n$-covered by irreducible curves $C$ such that $(D\cdot C)=\delta\geq n-1$.
It is well known that  a $X(r+1,2,1)\subset\p^N$ is necessarily  a $\p^{r+1}$ linearly embedded in $\mathbb P^N$ and that a non-degenerate $X(r+1, n, n-1)\subset\p^{n+r-1}$ is a variety of minimal degree $n-1$.
The smooth $X(r+1, 2,2)$'s  were recently  classified in \cite{ionescurusso} (see also \cite{Paterno} for a generalization to the polarized case). 
Without some reasonable restrictions  the classification of varieties $X(r+1, n, \dd)\subset\p^N$ becomes immediately  extremely difficult and out of reach, especially for singular varieties.
\smallskip 

Recently (see \cite{trepreauW}, \cite{piriotrepreau} and also \cite{pereirapirio}), it has been realized that the study of these varieties   is also closely related to  an important question in web geometry, namely  the algebraization 
of webs of maximal rank. In order to solve this problem of web geometry,  it was proved in  \cite{PT}
that the dimension of the linear span of such varieties satisfies the inequality 
\begin{equation}\label{bd}
\dim\big(\langle {X}(r+1,n,\dd)\rangle\big)  \leq \overline{\pi}(r,n,\delta)-1,
\end{equation}
 see \S 1. Here 
   $\overline{\pi}(r,n,\dd)={\pi}(r,n,d)$ with $d=\delta+r(n-1)+2$ where  
${\pi}\big(r,n,d\big)$ stands for  the  Castelnuovo-Harris  bound function  for the geometric genus of  nondegenerate irreducible $r$-dimensional projective varieties $Y\subset\mathbb P^{n+r-1}$ of degree $d$, see  \cite{harris} and  also Section \ref{castelnuovovar} for some  details
and  definitions. The bound \eqref{bd} is proved 
geometrically via  the iteration of projections from general osculating spaces to $X(r+1, n, \dd)$ determined
by the irreducible  curves of degree $\dd$ which $n$-cover the variety. This is  a classical tool used also by del Pezzo, Enriques and Castelnuovo to bound the dimension of linear systems on surfaces, see \cite{Ca}, \cite{Ca2}, \cite{En}, \cite[\S 7]{cilibertorusso} and  \cite{PT}. The above bound  also  reveals  a  connection
with  Castelnuovo theory of linear systems on curves and with the so called Castelnuovo varieties, 
see \cite{harris} and \cite{ciliberto}. 
From this point of view, non-degenerate varieties $X={X}(r+1,n,\dd)\subset\p^{
\overline{\pi}(r,n,\delta)-1}$, denoted from now on by  $X=\overline{X}(r+1, n, \delta)$,  are the extremal ones and they are subject to  very strong restrictions -- e. g. they are rational and through $n$ general points there passes a unique rational normal curve of degree $\dd$, see \cite{PT} and 
Theorem \ref{T:strongbompiani} and Theorem \ref{T:Siena} below.
  Due to these numerous geometrical properties, it is possible  in many cases to obtain a complete classification, see for example \cite{PT} or  Theorem \ref{T:strongbompiani}, Theorem \ref{smooth33} , Corollary \ref{class333}
and Corollary \ref{class433} below in this paper. \smallskip 

 Some basic results of \cite{PT} are generalized  here in Theorem \ref{T:Siena}. One proves  the bound $h^0(\O_X(D))\leq\overline\pi(r,n,\dd)$ for a Cartier divisor $D$
on a proper irreducible variety $X$ of dimension $r+1$,  $n$-covered by irreducible curves $C$ such that
$(D\cdot C)=\dd\geq n-1$. We also show that  equality holds  if and only if $\phi_{|D|}$ maps $X$ birationally onto a $\overline X(r+1,n,\dd)$. 
Furthermore, 
if $h^0(\O_X(D))=\overline\pi(r,n,\dd)$ then $X$ is rational and through $n\geq 2$ general points there passes
a unique smooth rational curve $C$ such that $(D\cdot C)=\dd$. 

Another consequence of the previous bound is that  under the same hypothesis we have  $D^{r+1}\leq\dd^{r+1}/(n-1)^r$  if   $D$ is nef,  see Theorem \ref{P:bounddeg}. This is  a generalization of a result usually attributed to Fano in the case  $n=2$, see for example  \cite[Proposition V.2.9]{kollar}.
\smallskip 

The classical roots of this type of results go back to C. Segre, \cite{segre},  who proved that  $\dim(\langle X(2,2,2)\rangle)\leq 5$ and that $\overline X(2,2,2)\subset\p^5$ is projectively equivalent to the Veronese surface. Bompiani
generalized this result
in \cite{bompiani} to  $\dim(\langle X(r+1, 2, \dd)\rangle)\leq { r+1+\dd\choose r+1}-1$ with equality holding if and only if  $\overline X(r+1, 2,\dd)$ is projectively equivalent to the $\dd$--Veronese embedding of $\p^{r+1}$,
see Theorem \ref{T:strongbompiani} here and also \cite{ionescuSiena} and \cite{trepreau}. 

A lot of examples of $\overline{X}(r+1, n,\delta)$ (for arbitrary $n\geq 2$, $r\geq 1$ and $\delta\geq n-1$) have been described in  \cite{PT} via the theory of Castelnuovo varieties and their construction will be briefly recalled in \S \ref{castelnuovovar}. The main result of \cite{PT}
ensures  that these examples {\it of Castelnuovo type} are the only ones except possibly when $n>2$, $r>1$ and $\dd=2n-3$.

Here we  consider in detail the last open case, that is the classification of  varieties $\overline{X}(r+1, n, 2n-3)$, especially for $n=3$.  We immediately point out that there are examples of $\overline{X}(r+1, n, 2n-3)$ of dimension at least 3 that are not of Castelnuovo type. For $n=3$ these varieties
 share very  special structures being related to the theory of cubic Jordan algebras, see \S \ref{S:JordanAlgebra}. Indeed interesting examples of $\overline X(r+1,3,3)\subset\p^{2r+3}$ are the so called  {\it twisted cubics over  complex Jordan algebras of rank 3},
 see \cite{mukai}. There is an infinite family of such varieties:  the Segre embeddings $\mathbb P^1\times Q^r\subset\p^{2r+3}$,  where $Q^r\subset\p^{r+1}$ is an irreducible hyperquadric, and also four smooth exceptional examples associated to the four  simple 
cubic Jordan algebras (these four varieties are also  known as {\it Lagrangian Grassmannians}) and other examples constructed by considering cubic Jordan algebras naturally arising from associative algebras, see \S \ref{S:JordanAlgebra}. For an arbitrary  $\overline X(r+1,3,3)\subset\p^{2r+3}$ we consider the birational projection onto $\p^{r+1}$ from a general tangent space. By studying the geometry
of this birational map  we are able to give an explicit parametrization of these varieties and also to associate  to them
a {\it quadro-quadric} Cremona transformation  from  the projectivization of a general affine tangent space onto a hyperplane in $\p^{r+1}$, see Theorem \ref{Cremonascroll}. From this unexpected connection we  deduce the classification
of arbitrary $\overline X(r+1, 3,3)$ for $r\leq 3$ (even if our method actually works also for $r=4$ or for bigger values of $r$), see Corollary \ref{class333} and Corollary \ref{class433}. As a consequence we also prove that the base locus of a quadro-quadric
Cremona transformation and the base locus of its inverse are projectively equivalent so that essentially these transformations
are involutions, Corollary \ref{involutions}, \ro[ a fact which seems to have been overlooked  as far as we know].
Moreover in Theorem \ref{smooth33} we provide the classification of all smooth $\overline X(r+1,3,3)$, showing that they  are either smooth rational normal scrolls (hence  of Castelnuovo type) or the Segre embeddings of $\p^1\times Q^r$ or one of the  four  Lagrangian Grassmannians. Our approach yields also a geometrical direct proof of the classification of all cubic Jordan algebras whose associated variety is smooth, showing that they are either simple (Lagrangian Grassmannian) or semi-simple ($\p^1\times Q^r$ with $Q^r$ smooth). 
\medskip

The paper is organized as follows. In \S \ref{S:Preliminaries and notations} we introduce some definitions and explain the notation. We also recall the main steps for the proof of the bound \eqref{bd}.
 In \S \ref{S:rationality},   a modern version of Bompiani's theorem \cite{bompiani} is proved, Theorem \ref{T:strongbompiani}. 
 Then, inductively via osculating projections 
 and the study of the rational map $\phi_{|D|}$ we  prove in Theorem \ref{T:Siena} the bound on the dimension of linear systems of Cartier divisors described above  and the consequences of the maximality of this dimension. 
In \S \ref{S:bounddegree} we deduce from the bound on the dimension of the linear span a general bound for the degree of nef divisors on varieties $n$-covered by irreducible curves.
 In \S \ref{S:examples},
after describing  the $\overline{X}(r+1, n, \dd)$ of Castelnuovo type in Section \ref{castelnuovovar},  
 we construct  new examples  of $\overline{X}(r+1, n, \dd)$ when $r>1$, $n>2$ and $\dd=2n-3$. In particular we describe in detail examples of $\overline{X}(r+1,3,3)$ of {\it Jordan type} in Section \ref{S:JordanAlgebra}.  In Section \ref{S:otherexamples} we present some examples of $\overline{X}(3,n, 2n-3)$ which  are not of Castelnuovo type for  $n=4,5,6$.
Section \ref{S:3-covered} concerns the classification of several classes of $\overline{X}(r+1,3,3)$
under different assumptions either on $r$,  Corollary \ref{class333} and Corollary \ref{class433},
 or on the smoothness of the variety, Theorem \ref{smooth33}. We also discuss some open problems pointing towards the equivalence of these apparently  unrelated objects: varieties $\overline X(r+1,3,3)$, quadro-quadric Cremona transformations of $\p^r$ and 
complex Jordan algebras of rank three. 
\smallskip

\hspace{-0.4cm}{\bf Acknowledgements.}
Both authors are grateful to Ciro Ciliberto for some discussions at different stages of the preparation of the paper and
for a lot of suggestions leading to an improvement of the exposition. The first author has considered the
problem studied here when developing a research in common with Jean--Marie Tr\' epreau. He 
learned a lot from the numerous discussions with him on this subject  during last years. The second author expresses his gratitude to Paltin Ionescu for a direct or indirect but undoubtedly rich intellectual exchange of ideas and points of view on some contents of the paper, for his interest in the results
and for a lot of remarks which improved the presentation.

\section{Preliminaries and notation}\label{S:Preliminaries and notations}

We shall consider irreducible varieties $X$ which are projective, or proper, over the complex field $\mathbb C$.

 We will use the following  notation: $r,n$ and $ \dd$ are positive integers such that $n-1\leq \dd$. 
 Then one defines $\rho=\big\lfloor  \frac{\dd}{n-1}\big\rfloor$ and $\epsilon=\dd-\rho (n-1)\in \{0,\ldots,n-2\}$. One also defines  $m=\epsilon+1>0$ and  $m'=n-1-m\geq 0$ so that  $m+m'=n-1$. Finally, for any integer  $k$, one sets  $k^+=\max\{ 0, k\}$. 

For classification results there is no restriction in supposing that
an irreducible variety $X\subset\p^N$  is non--degenerate. Otherwise   $\langle X \rangle \subset\p^N$ will denote  the {\it linear span of $X$ in $\p^N$}, that is the smallest linear subspace of $\p^N$ containing $X$. 
For computational reasons, when dealing with classification results,  we shall define $r$ such that $\dim(X)=r+1$ .

Let $x$ be a smooth point of $X$. For any $\ell\in \mathbb N$, we denote by $Osc^\ell_X(x)\subset \mathbb P^N$ the  {\it $\ell$th-order osculation space of $X$ at $x$}. If  $\psi:(\mathbb C^{r+1},0)\rightarrow  (X,x), u \mapsto \psi(u)$
 is a regular local parametrization of $X$ at $x$, then $Osc^\ell_X(x)$ can be defined as the projective subspace $\big\langle \partial^{|\alpha|}\psi(0)/\partial u^\alpha 
\, | \, \alpha\in \mathbb N^{r+1 },\, |\alpha|\leq \ell
\big\rangle
\subset  
\mathbb P^N$. This space can also be defined more abstractly as the 
linear subspace spanned by the $\ell$-th order infinitesimal neighborhood of $X$ at $x$ and also generalized to the case of arbitrary Cartier divisors on $X$.
Indeed, for every integer $\ell\in\mathbb N$, let $\mathcal P^\ell_X(D)$  denote the $\ell$-th principal part bundle (or $\ell$-th jet bundle) of  $\O_X(D)$. For every
linear subspace $V\subseteq H^0(X,\O_X(D))$ we have a natural homomorphism of sheaves
\begin{equation}\label{evjet}
\phi^\ell : V\otimes\O_X\to \mathcal P^\ell_X(D), 
\end{equation}
sending a section $s\in V$  to its $\ell$-th jet $\phi^{\ell}_x(s)$ evaluated at $x\in X$,
that is $\phi^{\ell}_x(s)$ is represented in local coordinates by the
Taylor expansion of $s$ at $x$, truncated after the order $\ell$. Taking a smooth point $x\in X\subset\p^N=\p(V)$ (Grothendieck's notation) and $\O_X(D)=\O_X(1)$, it is easily  verified that  $Osc_X^{\ell}(x)=\p(\im(\phi^{\ell}_x)).$
If $x\in X$ is a smooth point,  the previous definitions yield $\dim\, (Osc^\ell_X(x)) \leq 
\rk(\mathcal P^{\ell}_X)-1={r+1+\ell  \choose  r+1}-1$ and in general it is  expected that  equality holds at general points of $X\subset\p^N$
as soon as $N\geq {r+1+\ell  \choose  r+1}-1$. In this case, we shall say that the {\it osculation of order $\ell$ of $X$ at $x$ is regular}. 

A curve $C\subset \mathbb P^N$ is a {\it rational normal curve of degree $\dd$} if it is rational, smooth, of degree $\dd$ and its linear span in $\p^N$ has dimension $\dd$. In other terms: the restriction of $| \mathcal O_{\mathbb P^N}(1)|$ to $C$ is the complete linear system of degree $\dd$ on $C\simeq \mathbb P^1$. 

Let us define the following Castelnuovo-Harris function which  bounds the geometric genus of irreducible projective varieties, see Theorem \ref{P:boundCH} below:
\begin{equation}
\label{E:castelnuovoHarrisBound}
{\pi}(r,n,d)=\sum_{ \sigma \geq 0} {\sigma+r-1   \choose \sigma} \Big( d-(\sigma+r)\,(n-1)-1     \Big)^+.
\end{equation}

We will also use the following function that is closely related to ${\pi}(r,n,d)$: 

\begin{equation}
\label{F:defpibar}
\overline{\pi}(r,n,\delta):= m \, {r+\rho+1 \choose r+1 } +m'\, {r+\rho \choose r+1},
\end{equation}
  for $\delta\geq n-1$ fixed with $\rho=\lfloor \frac{\delta}{n-1} \rfloor$, $m=\dd-\rho\,(n-1)+1$ and  $m'=n-1-m$.

It is not difficult to prove that  $\overline{\pi}(r,n,\delta)=\pi(r,n,d)$, see \cite{PT} for details.

An irreducible projective variety $X\subset\p^N$ of dimension $r+1$ that is $n$-covered by a family of irreducible curves of degree $\dd$  will be denoted
by $X(r+1,n,\dd)\subset\p^N$. In most of the  cases we shall also assume that $X(r+1,n,\dd)\subset\p^N$ is  non-degenerate.
\bigskip

 For  reader's convenience we reproduce here some basic results of \cite{PT} on  varieties $X=X(r+1,n,\delta)\subset\p^N$.
For an irreducible curve $C \subset \mathbb P^{N}$  
 of degree $\dd$,  for non-negative integers  $a_1,\ldots,a_\kappa$, with $\kappa>0$ fixed and  such that $\sum_{i=1}^\kappa (a_i+1)\geq \dd+1$ and for $x_1,\ldots,x_\kappa\in C$ pairwise distinct smooth points,  one has:
\begin{equation}
\label{E:oscspanPn}
\big\langle C \big\rangle = \big\langle    
Osc^{a_i}_C(x_i) \, \big| \; i=1,\ldots, \kappa
\big\rangle
\end{equation}
(otherwise there would exist a hyperplane $H\subsetneq\langle C\rangle$ containing $\langle    
Osc^{a_i}_C(x_i) \, \big| \; i=1,\ldots, \kappa\rangle$ and $\dd=\deg(C)=\deg(H\cap C)\geq \sum_{i=1}^\kappa (a_i+1)$, contrary to our assumption).

Let $X=X(r+1,n,\dd)\subset\p^N$ and let $\Sigma$ be a fixed $n$-covering family of irreducible
curves of degree $\dd$ on $X$. 
 If  $x_1,\ldots,x_{n-1}$ are distinct general points on $X$ one can consider  the subfamily 
$ \Sigma_{x_1,\ldots,x_{n-1}}=\{ C \in \Sigma \, | \, x_i \in C \; \mbox{ for } \, i=1,\ldots,n-1\} $. 
Since $\Sigma$ is $n$-covering, the family $\Sigma_{x_1,\ldots,x_{n-1}}$ covers $X$ and we can also assume that the general curve in this family is non-singular at $x_1,\ldots, x_{n-1}$.
Let $\{a_1,\ldots,a_{n-1}\}$ be a set of  $n-1$  non-negative integers such that $\sum_{i=1}^{n-1} (a_i+1)\geq  \dd+1$. By  \eqref{E:oscspanPn} 
and since $Osc^{a_i}_C(x_i)\subset Osc^{a_i}_X(x_i)$ for every $i=1,\ldots,n-1$, it comes that 
$ \langle C \rangle \subseteq  \langle    
Osc^{a_i}_X(x_i) \, | \; i=1,\ldots,n-1 \rangle$
for a general 
$C\in \Sigma_{x_1,\ldots,x_{n-1}}$.
Since the elements of $\Sigma_{x_1,\ldots,x_{n-1}}$ cover $X$, one obtains 
\begin{equation}
\label{E:<X>}
 \langle X \rangle =  \big\langle    
Osc^{a_i}_X(x_i) \, \big| \; i=1,\ldots,n-1 \big\rangle .
\end{equation}
Therefore for these varieties we deduce that 
$
 {\dim}\,(\langle X \rangle) +1\leq \sum_{i=1}^{n-1} {r+1+a_i \choose r+1 } 
$. Taking $a_1=\cdots=a_m=\rho$ and $a_{m+1}=\cdots=a_{n-1}=\rho-1$  and recalling \eqref{F:defpibar}, 
we obtain  the following result for arbitrary $X=X(r+1,n,\dd)$, see \cite{PT}:
\begin{equation}
\label{F:bounddimINTRO}
{\dim}\, \big( \big\langle X \big\rangle  \big)  \leq 
\overline{\pi}\big(r,n,\delta\big)-1.
\end{equation}

Recall that for 
$\overline{\pi}(r,n,\delta)$ defined  in (\ref{F:defpibar}) we have $\overline{\pi}(r,n,\delta)=\pi(r,n,d)$, where the Castelnuovo-Harris bound $\pi(r,n,d)$ is defined in (\ref{E:castelnuovoHarrisBound}) and
where $d$ is defined  as a function of  $\dd$ by $d=\dd+r(n-1)+2$.
Since ${\rm dim}(Osc^\ell_X(x))\leq {r+1+\ell \choose r+1} -1$ for any point $x\in X$ and for any integer $\ell\in \mathbb N$,  we deduce an immediate consequence of \eqref{E:<X>} that for a non--degenerate   $X=X(r+1,n,\dd)\subset\p^{\overline{\pi}(r,n,\dd)-1}$ the following hold:
\begin{enumerate}
\item[{\rm(i)}] the osculation of order $\rho$ of $X$ at a general point $x\in X$ is regular, that is  
\begin{equation}\label{E:dimoscmax}
\dim  \big(Osc^\rho_X(x) \big)= {r+1+\rho \choose r+1} -1\,;
\end{equation}
\item[{\rm(ii)}] if $x_1,\ldots,x_{n-1}$  are general points of $X$, then
\begin{equation}
\label{E:decompoPpi2}
 \langle X \rangle =
\Big(\oplus_{i=1} ^m\,
Osc_{{X}}^{\rho}(x_i)\,\Big) \oplus 
\Big( \oplus_{j=1} ^{m'}\,
Osc_{{X}}^{\rho-1}(x_{m+j})\,\Big)={\mathbb P}^{\overline{\pi}(r,n,\dd)-1}.
\end{equation}
\end{enumerate}
\smallskip

From now on an irreducible non-degenerate projective variety $X=X(r+1,n,\dd)\subset\p^{\overline{\pi}(r,n,\dd)-1}$ 
will be denoted by $\overline X(r+1,n,\dd)\subset\p^{\overline{\pi}(r,n,\dd)-1}$ or simply by $\overline X(r+1,n,\dd)$. In the next sections we shall describe
the notable geometric properties of these varieties and of their covering families.

\section{Rationality of $\overline X(r+1,n,\dd)$ and of  the general curve of the  $n$-covering family}\label{S:rationality}

The following  simple remark, which is surely well known to the experts, will play a central role  several times in our analysis, see also  \cite[Lemma 2.2]{ionescuSiena} and \cite{PT} for  related statements. Since we were unable to find a precise reference for the generality needed, we also include a proof.
\medskip

\begin{lemma}
\label{L:Siena}
Let $\phi:X\map X'$ be a dominant rational map  between proper varieties of the same dimension,
let $\Sigma$ be an irreducible $n$-covering family of irreducible curves on $X$ and let $\Sigma'$
be the induced $n$-covering family on $X'$. If $X'$ is projective, if the restriction of $\phi$ to a general curve $C\in\Sigma$
induces a morphism birational onto its image and if through $n$-general points of $X'$ there passes a unique
curve $C'\in\Sigma'$, then the same is true for $\Sigma$ on $X$ and moreover $\phi$ is a birational
map.
\end{lemma}
\begin{proof} There exists a desingularization $\alpha:\widetilde X\to X$ with $\widetilde X$ projective and a morphism $\widetilde
\phi:\widetilde X\to X'$ solving the indeterminacies of $\phi$.  Thus without loss of generality we can assume $X$ smooth and projective, that $\phi$ is a morphism and that $\phi$ restricted to a general $C\in\Sigma$
is a morphism birational onto its image. The morphism $\phi$ is generically \' etale by generic smoothness, i.e. there exists an open set $U'\subseteq X'$ such that letting $U=\phi^{-1}(U')$, then $\phi_{|U}:U\to U'$ is an \' etale morphism. Let $d=\deg(\phi)=\deg(\phi_{|U})\geq 1$. We shall prove that $d=1$.

Fix $x_1,\ldots, x_{n-1}\in U$ general points. There exists an open subset $U_1\subseteq U$ such that
for every $x\in U_1$,  there passes exactly $s\geq 1$ curves in $\Sigma$ through  $x_1,\ldots, x_{n-1}$ and $x$.
Since $\phi$ is a proper morphism we can also take $U_1=\phi^{-1}(U'_1)$ with $U'_1\subseteq U'$ open.
Let $x'\in U'_1$ be a (general) point and let $\phi^{-1}(x')=\{\widetilde x_1,\ldots, \widetilde x_d\}$. Let $x'_l=\phi(x_l)$, 
$l=1,\ldots, n-1$,  and let  $\widetilde C_{i,1},\ldots, \widetilde C_{i,s}$ be the curves of $\Sigma$ passing through $x_1,\ldots, x_{n-1}, \widetilde x_i$. Then for every $j=1,\ldots, s$,
the curves $\phi(\widetilde C_{i,j})$ belong to $\Sigma'$ and  they pass through $x'_1,\ldots,x'_{n-1},x'$
so that   they coincide with the unique curve  $C'\in \Sigma'$ having this property. Since $\phi$ is a local isomorphism near $\widetilde x_i$, we deduce  $s=1$. In particular  through $n$ general points of $X$ there passes a unique curve belonging to  $\Sigma$.

Then $\widetilde x_i\in \widetilde C_i$, $x_1\in \widetilde C_i$ and $\phi(\widetilde C_i)=C'$ for every $i=1,\ldots, d$. Since $\phi$ is also a local isomorphism at $x_1$ since $x_1\in U$, we see that $\widetilde C_1=\ldots=\widetilde C_d=C$. Since $C\in \Sigma$ is general,  by hypothesis ${\varphi=\phi_{|C}}:C\to C'$ is a morphism birational onto its image, yielding  $d=1$ because $\widetilde x_i\in C$, $\phi(\widetilde x_i)=x'$ for every $i=1,\ldots, d$, and by the generality of $x'$ we can also suppose that ${\varphi^{-1}}(x')$ consists only of a point.
\end{proof}

We begin to study general properties of $\overline X(r+1,n,\dd)$'s starting  from the case $n=2$. This case was classically considered by Bompiani in \cite{bompiani}, where  the proof was essentially provided for surfaces. Under the assumption that  the general 2-covering curve is  smooth and rational, this  result was  also obtained by Ionescu in \cite[Theorem 2.8]{ionescuSiena}. A  similar but stronger  result holding in the  analytic category and
not only for complex  algebraic varieties has been proved recently by Tr\' epreau in \cite{trepreau}\footnote{Tr\'epreau's version of Bompiani's theorem is the following: {\it let $(X,x)\subset \mathbb P^N$ be a smooth germ of  a $(r+1)$-dimensional analytic variety with regular $\dd$th order osculation (at $x$).   Assume that $X$ is 
1-covered by (germs of)  rational curves of degree $\dd$ passing through $x$ that are (generically) smooth at this point. Then  $X\subset {\nu}_\dd(\mathbb P^{r+1})$}.}. 
\medskip

\begin{thm}
\label{T:strongbompiani}
 An  irreducible projective variety $X=\overline{X}(r+1,2,\dd)\subset\p^{\binom{r+1+\dd}{r+1}-1}$ is projectively equivalent to the Veronese manifold $\nu_\dd(\mathbb P^{r+1})$.
In particular every curve in  the $2$-covering family  is a rational normal curve of degree $\dd$ in the given embedding and there exists a unique such curve passing through two  points of $X$.
\end{thm}
\begin{proof} By definition $\rho=\dd$ so that by  \eqref{E:dimoscmax}, for $x\in X$ general we have  $$\dim(Osc_X^{\dd-1}(x))={r+\dd \choose r+1} -1
\qquad \text{ and } \qquad 
 Osc_X^{\dd}(x)=\p^{\binom{r+1+\dd}{r+1}-1}\,.$$ 
Now  let $x\in X$ be a fixed general point and let $T= Osc_X^{\dd-1}(x)=\p^{{r+\dd \choose r+1} -1}$. 
Let $p_T:X\map \p^{\overline \pi(r-1,2,\dd)-1}$ be the restriction to $X$ of the projection
from $T$. The rational map $p_T$ is given by the linear system $|D_x|$ cut on $X$ by hyperplanes containing $T$ so that the corresponding hyperplane
sections have a point of multiplicity $\dd$ at $x\in X$. A general irreducible  curve of degree $\dd$ passing through $x$ is  thus contracted by $p_T$. Let  $X_T=\overline{p_T(X)}$.  We have $\langle X_T\rangle=\p^{\overline \pi(r-1,2,\dd)-1}$ since  $\langle X\rangle=\p^{\overline \pi(r,2,\dd)-1}$. 

We claim that $X_T$ is projectively equivalent to $\nu_\dd(\p^r)\subset\p^{\overline\pi(r-1,2,\dd)-1}$.  
Indeed, let $\pi:\Bl_x (X)\to X$ be the blow-up of $X$ at $x$, let $E=\p^r$ be the exceptional divisor and let $p'_T=p_T\circ\pi:\Bl_x(X)\map X_T$ be
the induced rational map. The restriction of $p'_T$ to $E$ is a rational dominant map from $\p^r$ to $X_T\subset\p^{\overline\pi(r-1,2,\dd)-1}$ given by a sublinear system of $|\O_{\p^r}(\dd)|$ of dimension $\overline \pi(r+1,2,\dd)-1$ so that it is given by the complete linear system $|\O_{\p^r}(\dd)|$
(since $Osc_X^\delta(x)=\p^{\overline\pi(r,2,\dd)-1}$, the restriction of the strict transform of the linear system of  hyperplane sections containing $Osc_X^{\dd-1}(x)$ to $E$ is not zero). Thus the restriction
of $p'_T$ to $E$ induces an isomorphism between $E$ and $X_T$ given by $|\mathcal O_{\p^r}(\dd)|$, proving the claim. Moreover since  a general curve  $C\in\Sigma$ is not contracted by $p_T$, we have that $p_T(C)$ is a curve on $X_T$ of degree $\dd'\leq \dd$. Thus $p_T(C)$ is a smooth rational
curve of degree $\dd$,  $T\cap C=\emptyset$,  the rational map  $p_T$ is defined along  $C$ and it gives  an isomorphism between $C$ and $p_T(C)$.

By solving the indeterminacies of $p'_T$, we can suppose that there exists a smooth
variety $\widetilde X$, a birational morphism $\phi:\widetilde X\to X$ and a morphism $\widetilde p_T:\widetilde X\to X_T\simeq\nu_\dd(\p^r)$ such that $p_T\circ\phi=\widetilde p_T$. Let $\phi^*(|D_x|)=F_x+|\widetilde D_x|$ with $|\widetilde D_x|$ base point free and let 
$|\overline D_x|=\widetilde p_T^*(|\O_{\p^r}(1)|)$. Then $\widetilde D_x\sim \dd \overline D_x$ and $\dim(|\overline D_x|)\geq r$. Moreover for the strict transform of a general curve $C_x$  in $\Sigma$ passing through $x$ we have 
$(\widetilde D_x\cdot C_x)=0$ and $(F_x\cdot C_x)=\dd$ while for the strict transform of a general curve $C\in \Sigma$ we have $(\widetilde D_x\cdot C)=\dd$. Thus $(\overline D_x\cdot C_x)=0$ and
$(\overline D_x\cdot C)=1$. 

Letting $T'=Osc_X^{\dd-1}(x')$ with $x'\in X$ general and performing the same analysis we can also suppose that on $\widetilde X$ the rational map $p_{T'}\circ\phi=\widetilde p_{T'}$ is defined and that
there exists a linear system $|\overline D_{x'}|$ such that $\dim(|\overline D_{x'}|)\geq r$, 
$(\overline D_{x'}\cdot C_{x'})=0$ for general $C_{x'} \in \Sigma_{x'} $ and  $(\overline D_{x'}\cdot C)=1$ for  general $C\in \Sigma$. Since a general $C_{x'}$ in  $ \Sigma_{x'} $ does not pass
through $x$ we have $|\overline D_x|\neq |\overline D_{x'}|$. On the other hand for $x\in X$ general, the linear systems
$|\overline D_x|$ vary in the same linear system $|D|$ on $\widetilde X$ since $\widetilde X$ is rationally connected. 

Thus $\dim(|D|)\geq r+1$,  $(D\cdot C)=1$ for the strict transform of a general curve $C$ in $\Sigma$
and $C$ does not intersect the base locus of $|D|$ by the previous analysis. Let $s+1=\dim(|D|)$ and let  $\psi=\psi_{|D|}:\widetilde X\map \widetilde X'\subseteq\p^{s+1}$ be the associated rational  map. Since $\psi(C)$ is a line passing through two general points of $X'$, we deduce  $X'=\p^{s+1}$  and $r=s$. Moreover by Lemma \ref{L:Siena} the rational map $\psi$ is birational.  Hence  there exists a birational map $\varphi=\phi\circ\psi^{-1}: \mathbb P^{r+1} \dashrightarrow X$ sending a general line in $\p^{r+1}$ onto
a general curve of degree $\dd$ in $\Sigma$. Composing
$\varphi$ with the inclusion $ X \subset \mathbb P^{\overline \pi(r,2,\dd)-1}$ 
 we get a birational map from $\mathbb P^{r+1}$  given by a
sublinear system of $ | \mathcal O_{\mathbb P^{r+1}}(\dd)|$ of dimension
$  {r+1+\dd \choose r+1}-1$, that is $\varphi$ is  given by the complete linear
system  $ | \mathcal O_{\mathbb P^{r+1}}(\dd)|$. In conclusion $X\subset\p^{\overline \pi(r,2,\dd)-1}$
 is projectively equivalent to the Veronese manifold $\nu_\dd(\mathbb P^{r+1})$.
\end{proof}
\medskip

The rationality and the smoothness of the general member $\Sigma$ of the 2-covering family of a $\overline X(r+1,2,\dd)$ can also be deduced differently. Indeed
in the previous proof we saw that  the linear system of hyperplane sections having a point
of multiplicity greater than or equal to $\dd$  at a general $x\in X$ cuts a general $C\in\Sigma_ x$ in the Cartier divisor $\dd x$. By varying $x$ on $C$ we see that
this property holds for the general point of $C$. Thus the smoothness and rationality of a  general
element of $\Sigma$ are consequences
of the following classical and surely well known result, which seems to go back  to Veronese, \cite{veronese},  at least in the projective version. The proof is well known and left to the reader.

\begin{lemma}
\label{L:charRN}
Let $C$ be an irreducible projective curve. Then:
\begin{enumerate}

\item if $C\subset\p^N$ is  non-degenerate and of  degree $\dd$, then $N\leq\dd$ and the  following conditions are equivalent:
\begin{enumerate}
\item $N=\dd$ and $C\subset\p^{\dd}$ is a rational normal curve of degree $\dd$;

\item for a general $x\in C$ there exists a hyperplane $H_x\subset\p^N$ such that $H_x\cap X=\dd\cdot x$ as schemes.
\end{enumerate}
\medskip

\item The following conditions are equivalent:
\begin{enumerate}
\item[(a')] $C$ is a smooth rational curve;

\item[(b')] there exists a Cartier divisor $D$  of degree $\dd\geq 1$ on $C$ such that $\dim(|D|)=\dd$.

\item[(c')] $\O_C(\dd\cdot x_1)\simeq\O_C(\dd\cdot x_2)$  for some $\dd\geq 1$ and for $x_1,x_2\in C$ general points.

\end{enumerate}
\end{enumerate}
\end{lemma}
\medskip

Via Lemma \ref{L:charRN} one could also prove differently Theorem \ref{T:strongbompiani} above following the steps of Mori's characterization of projective spaces given in \cite{Mori} because  in this case the family of smooth rational curves  $\Sigma_x$ is easily seen to be proper. 
\medskip

We now investigate the higher dimensional versions of  Lemma \ref{L:charRN} from  the point of view of
proper varieties $n$-covered by irreducible curves of degree $\dd$ with respect to some fixed divisor $D$, providing generalizations
of  \cite[Theorem 2.8]{ionescuSiena} and of  \cite[Theorem 4.4]{BBI}. Part (1) below has been obtained  in \cite{PT} while (i) is an abstract version of  \eqref{bd}. 
\medskip

\begin{thm}\label{T:Siena} Let $X$ be an irreducible proper variety of dimension $r+1$ and let $D$ be a Cartier divisor on $X$. Suppose that through $n\geq 2$ general points
of $X$ there passes an irreducible curve $C$ such that $(D\cdot C)=\dd\geq n-1$. Then:
\begin{enumerate}
\item[{\rm (i)}]   $h^0(X,\mathcal O_X(D))\leq\overline\pi(r, n, \dd)$;
\medskip
\item[{\rm (ii)}] Equality holds in (i)  if and only if $\phi_{|D|}$ maps $X$ birationally onto a $\overline X(r+1, n, \dd)\subseteq\p^{\overline\pi(r,n,\dd)-1}$.
In this case the  general deformation of $C$ does not intersect the indeterminacy locus of $\phi_{|D|}$.

\item [{\rm (iii)}] If equality holds in (i), then 
\begin{enumerate}
\item the variety $X$ is rational;

\item  the general deformation $\overline C$ of $C$ is a smooth rational curve and through $n$ general points of $X$ there passes a unique smooth rational curve $\overline C$ such that $(D\cdot \overline C)=\dd$. 
\end{enumerate}
\end{enumerate}

In particular:
\begin{enumerate}
\item a $\overline X(r+1, n,\dd)\subset\p^{\overline\pi(r,n,\dd)-1}$ is rational, the general curve of the
$n$-covering family is a rational normal curve of degree $\dd$ and through $n$ general points of $X$ there passes
a unique rational normal curve of degree $\dd$;

\item  a $\overline X(r+1,n,\dd)\subset\p^{\overline\pi(r,n,\dd)-1}$ is a linear birational projection of $\nu_{\dd}(\p^{r+1})$ {\rm (}or equivalently, a $\overline X(r+1, n, \dd)$ is the birational image of $\p^{r+1}$
given by a linear system of hypersurfaces of degree $\dd$ and dimension $\overline\pi(r,n,\dd)-1${\rm )}.
\end{enumerate}
\end{thm}
\begin{proof} Suppose $h^0(\O_X(D))\geq\overline \pi(r,n,\dd)\geq 2$ and let $$\phi=\phi_{|D|}:X\map X'\subseteq\p(H^0(\O_X(D))=\p^N.$$ The variety $X'\subseteq\p^N$ is irreducible, non-degenerate,  of dimension $\dim(X')=s+1\leq r+1$ and is $n$-covered by irreducible curves of degree $\dd'\leq \dd$.
Therefore by \eqref{F:bounddimINTRO} one gets  
$$h^0(\O_X(D))\leq \overline\pi(s,n,\dd')\leq \overline\pi(r,n,\dd),$$
 yielding $h^0(\O_X(D))=\overline \pi(r,n,\dd)$ and $X'=\overline X(r+1,n,\dd)$. We have thus proved the bound (i) and also that if equality 
holds then $X'=\overline X(r+1,n,\dd)$.

One implication of the first part of (ii) is trivial and we shall prove only the non-trivial implication and the second part of (ii). If $h^0(\O_X(D))=\overline\pi(r,n,\dd)$, then by the previous analysis $\dim(X')=r+1$ and $\deg(\phi(\overline C))=\dd$ for a general deformation $\overline C$ of $C$.

Now we shall prove the birationality of $\phi$ and part (iii) for $n=2$. Later  we shall treat  the general case $n>2$.
If $n=2$, it  follows from Theorem 
\ref{T:strongbompiani} that $X'=\overline X(r+1,2,\dd)$ is projectively isomorphic to $\nu_{\dd}(\p^{r+1})$ so that $\phi(\overline  C)$ is the unique rational normal curve of degree $\dd$ passing through two general points of $X'$. Let $|\overline D|$ be the linear system on $\overline C$ obtained by restricting 
$|D|$ to $\overline C$. Then $\deg(\overline D)=\dd$ and $\dim(|\overline D|)=\dd$ since $\phi(\overline C)$ is a rational normal curve of degree $\dd$. Part (2) of Lemma \ref{L:charRN} implies that $\overline C$ is a smooth rational curve, that $\phi$ is defined along $\overline C$ and that the restriction of $\phi$ to $\overline C$ is an isomorphism onto its image. Thus  from  Lemma \ref{L:Siena} we deduce that $\phi$ is birational
and that through $2$ general points of $X$ there passes a unique smooth rational curve $\overline C$
such that $(D\cdot \overline C)=\dd$.

Suppose $n>2$ and recall  the following notation, see \S \ref{S:Preliminaries and notations}: 
$\rho=\lfloor {\delta}/{(n-1)} \rfloor$,  $ m=\dd-\rho\,(n-1)+1$ and $m'=n-1-m$. 
Let $x_1,\ldots,x_{n-1}$  be $n-1$ general points on  $X'=\overline X(r+1,n,\dd)\subset\p^{\overline\pi(r,n,\dd)-1}$.  By \eqref{E:dimoscmax} and \eqref{E:decompoPpi2}, we know that
the osculating spaces $Osc_{{X'}}^{\rho}(x_i)$ and $Osc_{{X'}}^{\rho-1}(x_{m+j})$ (for $i=1,\ldots,m$ and $j=1,\ldots,m'$) have the maximal possible dimension and are in direct sum in the ambient space, yielding a   decomposition 
$${\mathbb P}^{\overline\pi(r,n,\dd)-1}=\langle 
{X'}\rangle =  Osc_{{X'}}^{\rho}(x_1)   \oplus  S,$$  
where 
$$S=\big( \oplus_{i=2} ^m\,
Osc_{{X'}}^{\rho}(x_i)\,\big) \oplus 
\big( \oplus_{j=1} ^{m'}\,
Osc_{{X'}}^{\rho-1}(x_{m+j})\,\big).$$

Let $S$ be as above, let  
$$p_S:X'\map L=Osc_{{X'}}^{\rho}(x_1)=\p^{\overline\pi(r,2,\rho)-1}$$ be the restriction to $X'$ of the linear projection  from    $S$ onto $L$ and let $X'_S=\overline{p_S(X')}$.
Let  $\Sigma_{x_2,\ldots,x_{n-1}}$ be the family consisting of curves in the covering family $\Sigma$ passing through $x_2,\ldots, x_{n-1}$. The family  $\Sigma_{x_2,\ldots,x_{n-1}}$ is 2-covering and a general $C\in \Sigma_{x_2,\ldots,x_{n-1}}$
has contact of order at least $(m-1)(\rho+1)+m'\rho$ with $S$ and   is not contracted by $p_S$.
Thus  
the irreducible curve $C_S=p_S(C)$ has degree $\rho'\leq \dd-(m-1)(\rho+1)-m'\rho=\rho$. The projections of the curves in $\Sigma_{x_2,\ldots,x_{n-1}}$  produce a 2-covering family of irreducible  curves of degree $\rho'$ on the non-degenerate irreducible variety $X'_S\subset\p^{\overline\pi(r,2,\rho)-1}$.
Then  \eqref{F:bounddimINTRO} implies  that $X'_S=\overline X(r+1,2,\rho)$ is projectively equivalent to $\nu_\rho(\p^{r+1})$ so that  $C_S\subset X'_S$ is a rational normal curve of degree $\rho$ by Theorem \ref{T:strongbompiani}. Moreover, $C$ is a rational normal curve of degree $\dd$ by part (2) of Lemma \ref{L:charRN} since the restriction of $p_S$ to $C$ is given by a linear system of degree $\rho$ and $p_S(C)=C_S$ is a rational normal curve of degree $\rho$.

The dominant rational map $p_S$ is birational and through $n\geq 2$ points of $X'$ there passes a unique rational normal curve of degree $\dd$ by Lemma \ref{L:Siena}. Thus (1) is proved for every  $n\geq 2$.
Applying once again Lemma \ref{L:Siena} to $\phi:X\map X'$  we immediately deduce  also for $n>2$ that the map $\phi$ is birational and  that  through $n> 2$ points of $X$ there passes a unique smooth rational curve $\overline C$ such that $(D\cdot \overline C)=\dd$, proving (a) and (b) of (iii).

To prove (2), let $\varphi=p_S^{-1}:\p^{r+1}\map X'=\overline X(r+1,n,\dd)$. The birational map $\varphi$   sends  a general line $l\subset\p^{r+1}$ onto a general element $C\in \Sigma_{x_2,\ldots,x_{n-1}}$.
The composition of $\varphi$ with the natural inclusion  $X'=\overline X(r+1,n,\dd)\subset\p^{\overline\pi(r,n,\dd)-1}$ is  given by a sublinear system of $|\O_{\p^{r+1}}(\dd)|$ of dimension $\overline\pi(r,n,\dd)-1$, 
showing   that $X'=\overline X(r+1,n,\dd)$ is a birational linear projection of $\nu_\dd(\p^{r+1})$ from a linear space of dimension 
$\overline\pi(r,2,\dd)-\overline\pi(r,n,\dd)-1$.
\end{proof}
\medskip
The preceding statement concerns varieties $X$ that are $n$-covered by irreducible curves of degree $\dd$ with respect to an arbitrary  given Cartier divisor hence it is more general than 
 the  corresponding result  in \cite{PT} considering  embedded $X=\overline{X}(r+1,n,\dd)\subset \p^{\overline \pi(r,n,\dd)-1}$. Note however that the main tool used to prove Theorem \ref{T:Siena}, namely the reduction to the well understood case $n=2$ via an osculating projection, is  the same as in \cite{PT}.

\section{Bound for the top self intersection of a nef divisor}
\label{S:bounddegree}

In this section, as a consequence of part (i) of Theorem \ref{T:Siena},  we prove a bound for the top self intersection of a nef divisor $D$ on a proper
variety $X$ such that through  $n\geq 2$ general points there passes an irreducible curve $C$ with 
$(D\cdot C)=\dd\geq n-1$. In particular we obtain a bound for the degree of varieties $X(r+1,n,\dd)\subset\p^N$. The bound 
\eqref{F:bounddeggen} below generalizes a result usually attributed to Fano, who  proved it for $n=2$. The reader can consult   the modern reference \cite[Proposition V.2.9]{kollar} for the case $n=2$  of Fano's result and also the several  applications given in {\it loc. cit.}, e.g.  to the boundedness of  the number of components of families of smooth Fano varieties of a fixed dimension, see  \cite[Chap.\,V]{kollar}.
\medskip
 
\begin{thm}\label{P:bounddeg} 
Let $X$ be a proper irreducible variety of dimension $r+1$, let $D$ be a nef Cartier divisor on $X$ and suppose that through $n\geq 2$ general points there passes an irreducible curve $C$ such that $(D\cdot C)=\dd\geq n-1$. Then
\begin{equation}
\label{F:bounddeggen}
D^{r+1}  \leq  \frac{\dd^{r+1}}{(n-1)^r}\;.
\end{equation}

In particular, if    $X=X(r+1,n,\dd)\subset\p^N$, then
\begin{equation}
\label{F:bounddeg}
{\rm deg}( X)   \leq  \frac{\dd^{r+1}}{(n-1)^r}\; .
\end{equation}
\end{thm}
\begin{proof} By the Asymptotic Riemann-Roch Theorem, see for example \cite[Theorem VI.2.15]{kollar},
 we know that 
$$h^0(\O_X(\ell D))=D^{r+1}\frac{\ell ^{r+1}}{(r+1)!}+O(\ell^r)$$
so that 
\begin{equation}
\label{E:degdef}
D^{r+1}= \lim_{\ell \rightarrow +\infty} \frac{(r+1)!h^0(\O_X(\ell D) )}{\ell^{r+1}} \; . 
\end{equation}

Since $X$ is $n$-covered by a family of irreducible curves having intersection with $D$ equal to  $\delta$,  $X$
 is also $n$-covered by a family of irreducible   curves having intersection  $ \delta\ell $ with $\ell D$ for any $\ell>0$. Theorem \ref{T:Siena} yields 
\begin{equation*}
 h^0(\O_X(\ell D))\leq \overline{\pi}(r,n,\delta\ell)
\end{equation*}
for every positive integer $\ell$. From \eqref{E:degdef} we deduce 
\begin{equation}
 D^{r+1}\leq  \liminf_{\ell \rightarrow +\infty} \frac{(r+1)!\;\overline{\pi}(r,n, \delta\ell)}{\ell^{r+1}} \; . 
\end{equation}
 Let $\rho_\ell=\lfloor \frac{\dd \ell}{n-1} \rfloor$ for $\ell>0$. The definition of $\overline{\pi}(r,n,\delta\ell)$ in \eqref{F:defpibar} implies that 
\begin{equation*}
 D^{r+1}\leq  \liminf_{\ell \rightarrow +\infty} \frac{(n-1) \,\big(r+1+ \rho_\ell\big)!}{\ell^{r+1}\rho_\ell!} \; . 
\end{equation*}

Using Stirling's formula, for $\ell\to+\infty$, we have 
\begin{align*}
 \frac{(n-1) \,\big(r+1+ \rho_\ell\big)!}{\ell^{r+1}\rho_\ell!} & \sim
\frac{(n-1) \,\sqrt{ r+1+ \rho_\ell}\, \big(\frac{r+1+ \rho_\ell}{e} \big)^{r+1+\rho_\ell}}{\ell^{r+1} \sqrt{\rho_\ell} \big( \frac{\rho_\ell}{e}\big)^{\rho_\ell}} 
\\
& \sim 
\frac{(n-1) \, \big({r+1+ \rho_\ell} \big)^{r+1}}{\ell^{r+1} e^{r+1} } 
\, \Big( 
1+ \frac{r+1}{\rho_\ell } \Big)^{\rho_\ell}\, .
\end{align*}
Since  $\rho_\ell\rightarrow +\infty$ if $\ell\to+\infty$   and recalling that $\lim_{x\rightarrow +\infty} \big( 1+ \frac{r+1}{x}\big)^x=e^{r+1}$, we obtain
\begin{align*}
 \frac{(n-1) \,\big(r+1+ \rho_\ell\big)!}{\ell^{r+1}\rho_\ell!} & \sim
\frac{(n-1) \, \rho_\ell^{r+1}}{\ell^{r+1}  } \; . 
\end{align*}

But $\rho_\ell \sim  \frac{\delta \ell}{n-1}$ if $ \ell\rightarrow +\infty$ hence we finally get
\begin{equation*}
 D^{r+1}\leq  \liminf_{\ell \rightarrow +\infty} 
\frac{(n-1)\, \big( \frac{\delta \ell}{n-1}  \big)^{r+1}}{\ell^{r+1}  } 
=  \frac{\dd^{r+1}}{(n-1)^r}\; .\qedhere
\end{equation*}
\end{proof}

\begin{rem}
{\rm  The bound (\ref{F:bounddeg})  is sharp for $n=2$ since $\dd^{r+1}$ is the degree of $v_\dd(\mathbb P^{r+1})$. 
More generally, it is sharp for every $n\geq 2$ as soon as $\dd=\rho(n-1)$ for some integer $\rho$ since in this case $\dd^{r+1}/(n-1)^r
=\rho^{r+1}(n-1)$ is the degree of $v_\rho(Y)$ for any non-degenerate variety $Y^{r+1}\subset \mathbb P^{n+r-1}$ of minimal degree $n-1$.} 
\end{rem} 
\medskip

We apply the previous bound on the degree to classify the $\overline{X}(r+1, n, n-1+k)$'s for $n$ sufficiently large,  when $r$ and $k$ are fixed.

Suppose $k=0$.  Since $\overline{\pi}(r, n, n-1)=n+r$,   $X=\overline{X}(r+1,n,n-1)\subset\p^{n+r-1}$ is a variety of minimal degree equal to ${\rm codim}(X)+1=n-1$ by \eqref{F:bounddeg}, as it is well known.
These varieties were classified in ancient times by classical algebraic geometers, see \cite{eisenbudharris} and also \S \ref{S:examples}.

Now we consider the case  $k>0$. 
When $n$ is sufficiently large, we have  $m=k+1,$ $m'=n-k-2$ and $\rho=1$, yielding  
 $\overline{\pi}(r,n, n-1+k)=(k+1)(r+2)+n-k-2$. So \eqref{F:bounddeg} implies that $X=\overline{X}(r+1, n, n-1+k)$ is a variety of minimal degree as soon as 
the quantity 
\begin{equation}
\label{E:theta}
 \theta= (n-1+k)^{r+1}-(n-1)^r\big(n+k(r+1)-2 \big) 
\end{equation}
 is strictly positive. But $\theta= n^r+ O(n^{r-1})$ as a simple and direct expansion shows. 
Therefore for $n$ sufficiently large $\theta>0$ and   $X=\overline X(r+1,n,n-1+k)$ is a variety of minimal degree. Assuming moreover  that $n>5$, as we shall do from now on,  one deduces that $X$ is a rational normal scroll $S_{a_0,\ldots,a_r}$ with $0\leq a_0 \leq \ldots \leq a_r $  and  $\sum_{i=0}^r a_i=n+k(r+1)-1$. We want to prove that $X$ is projectively equivalent to  $S_{\alpha_0+k,\ldots,\alpha_r+k}$ with $\alpha_0,\ldots,\alpha_r$ verifying  $0\leq \alpha_0 \leq \ldots \leq \alpha_r $  and $\sum_{i=0}^r \alpha_i=n-1$. With the terminology introduced in \S \ref{S:examples} this will mean exactly that 
$X=\overline X(r+1,n,n-1+k)$ is of {\it Castelnuovo type} for $n$ sufficiently large.

 Let $x_1,\ldots,x_{k+1}$ and $y_1,\ldots,y_{n-k-2}$ be general points of $X$ and let  $\pi=\overline{\pi}(r,n,n-1+k)$. 
According to  \eqref{F:bounddimINTRO},
\begin{equation*}
{\mathbb P}^{\pi-1}=\big\langle 
{X}
  \big\rangle=
\Big( \oplus_{i=1} ^{k+1}\,
Osc_{{X}}^{1}(x_i)\,\Big) \oplus 
\big\langle y_1,\ldots,y_{n-k-2}
\big\rangle
 \, .
\end{equation*}

Let us introduce some notation. Let $0\leq a_0\leq a_1\leq \ldots\leq a_r$ with $a_r>0$ be
integers and set 
$\p(a_0,\ldots,a_r):=\p(\oplus_{i=0}^r {\mathcal O}_{\p^1}(a_i))$.
Let  $H$ be a divisor in $|{\mathcal O}_{\p(a_0,\ldots,a_r)}(1)|$ and
consider the morphism $\phi=\phi_{|H|}: \p(a_1,\ldots,a_n)\to  \p^{a_0+\cdots+a_r+r}$
whose image is denoted  $S_{a_0,\ldots,a_r}$ and called a {\it  rational
normal scroll of type $(a_0,\ldots, a_r)$}.
The morphism $\phi$ is birational onto its image so that
$S_{a_0,\ldots,a_r}$ has dimension $r+1$ and its degree is
$a_0+\cdots +a_r$. The  scroll $S_{a_0,\ldots,a_r}$ is smooth if and only if  $a_0>0$ and  $\phi$ is an embedding in this case. 
If $0 =a_i<a_{i+1}$, then $S_{a_0,\ldots,a_r}$ is a cone
over $S_{a_{i+1},\ldots,a_r}$ with vertex a $\p^{i}$.

\begin{lemma}[\cite{cilibertorusso} p.\,13]
 \label{L:cilibertorusso}
 Let $b_0,\ldots,b_r$ {be} natural integers such that $0=b_0=\cdots=b_i<b_{i+1} \leq b_{i+2}\leq \ldots \leq b_r$, let $\{d'_1,\ldots,d'_{r-i}\}=\{ b_{i+1}-1,\ldots,b_r-2\}$
 and let $\{d_1,\ldots,d_{r-i}\}$ be a rearrangement of $\{d'_1,\ldots,d'_{r-i}\}$ such that $0\leq d_1\leq d_2\leq\ldots\leq d_{r-i}.$ 
  Then the  general tangential projection of   $S_{b_0,\ldots,b_r}$ is   $S_{d_1,\ldots,d_{r-i}}$.
\end{lemma}

Thus  if $a_0<k$, there would exist $\ell<k+1$ such that the image $X'$ of $X$ via  the linear projection $p_S$ from  $S=
 Osc_{X}^1(x_1)\oplus \cdots {\oplus}Osc_{X}^1(x_\ell)$ would be  a rational normal scroll of dimension $r'\leq r$. This would imply 
 $\dim(\langle \overline X(r+1, n, n-1+k)\rangle)<\overline \pi(r,n,n-1+k)-1$, leading to a contradiction. In conclusion we proved the following consequence of \eqref{F:bounddeg}.

\begin{coro}  If $n$ is sufficiently large, a 
  $\overline{X}({r+1},n, n-1+k)$ is  projectively equivalent to  a rational normal scroll $S_{\alpha_0+k,\ldots,\alpha_r+k}$ with  $\alpha_0,\ldots,\alpha_r$   such that $\sum_{i=0}^r \alpha_i=n-1$. 
\end{coro}

\begin{rem}{\rm
 The value of $n$ in the previous result can be made effective: $\overline{X}(r+1, n, n-1+k)$
 is a scroll as in the previous corollary  as soon as $n \geq \max(6,k-2)$ and $\theta>0$,  where $\theta$ is the quantity defined in \eqref{E:theta}.}
\end{rem}

\section{Some examples of  varieties
$\overline{X}(r+1,n,\dd)$ }\label{S:examples}

In this Section we describe in detail three classes of  $\overline{X}(r+1,n,\dd)$. 
 The simplest ones are the so called  varieties of minimal degree.  
Next, using the theory of Castelnuovo varieties, we will construct examples 
of $\overline{X}(r+1,n,\dd)$  for arbitrary $n,r$ and $\dd\geq n-1$. These latter examples have already been presented in \cite{PT}. 
Finally, we will  show that twisted  cubic curves associated to  Jordan algebras of rank 3 are examples of $\overline{X}(r+1,3,3)$. 
 
\subsection{Varieties of minimal degree  and their associated models \cite{eisenbudharris,harris}}
\label{S:varmini}
It is well-known  (see \cite[p.173]{griffithsharris}) 
 that  the degree of an irreducible non-degenerate projective variety 
$Y \subset {\mathbb P}^{n+r-1}$ of dimension $r+1$ satisfies ${\rm deg}(Y)\geq {\rm codim}(Y)+1=n-1$. By definition,  $Y\subset\p^{n+r-1}$ is a {\it variety of minimal degree} if its degree is $n-1$. Such varieties exist and are well
known. Their classification goes back to Bertini and Enriques and can be summarized as follows  (the notation being as in  \cite{harris}):

 \begin{thm}
\label{varmini}  
The following is an exhaustive non redundant list of the $(r+1)$-dimensional varieties $Y\subset{\mathbb P}^{n+r-1}$ of minimal degree $n-1$: 
\begin{enumerate}
 \item
 the rational normal scrolls  $S_{a_0,\ldots, a_{r}}$ 
for some integers $a_0,\ldots, a_{r}$ verifying
$0 \leq a_0\leq a_1 \leq \ldots \leq a_{r} \leq n-1$\, and\,  $ a_0+\cdots+a_r=n-1$; \vspace{0.1cm} 
\item  the ambient space $\mathbb P^{r+1}$ itself  ($n=2$); \vspace{0.1cm} 
\item the quadric hypersurfaces of rank 
 $\varrho \geq 5$ ($n=3$);  \vspace{0.1cm} 
\item the cones over a Veronese surface in $\mathbb P^5$ ($n=5$). 
\end{enumerate}
\end{thm}

Recall that a scroll $S_{a_0,\ldots, a_{r}}$ is singular if and only if $a_0=0$. In this case, it is a cone over the scroll $S_{a_j,\ldots,a_{r}}$ where $j$ stands  for the smallest integer such that $a_j\neq 0$. 
 
Let $Y\subset \mathbb P^{n+r-1}$ be a variety of minimal degree. The space spanned   by $n$ generic distinct points $y_1,\ldots,y_n$ on $Y$ is a  
 $(n-1)$-dimensional subspace in $ \mathbb P^{n+r-1}$. The latter being generic, it intersects $Y$ along an irreducible non-degenerate curve of degree $n-1$, which by Lemma \ref{L:charRN} is a rational normal curve in $\langle y_1,\ldots,y_n\rangle$
 passing through $y_i$ for any $i=1,\ldots,n$. This shows that  $Y\subset\p^{n+r-1}$ is $n$-covered by rational normal curves of degree $n-1$. Since $\overline \pi(r,n, n-1)=n+r$, it follows that varieties of minimal degree are examples of  $\overline{X}(r+1,n,n-1)$,  which  will be called  {\it models of minimal degree}.  Note that also the converse is true because  according to Theorem  \ref{P:bounddeg}, every $\overline{X}(r+1,n,n-1)\subset\p^{n+r-1}$ is a variety of minimal degree $n-1$.

\subsection{Castelnuovo's varieties   and their associated models \cite{harris,ciliberto,PT}}\label{castelnuovovar} 
Let $V\subset\mathbb P^{n+r-1}$ be an irreducible non-degenerate variety of dimension $r$ and degree $d>1$. There is  an explicit bound on the {geometric genus} $g(V)$ of $V$  in terms of an explicit constant depending on $d,n$ and $r$. The {\it geometric genus} $g(V)$ is defined as  the dimension $h^0(K_{\tilde{V}})$ for one (hence for all)   resolution  of the singularities $\tilde{V}\rightarrow V$ of $V$.

\begin{thm}[Castelnuovo-Harris bound \cite{harris}]\label{P:boundCH}
The following bound holds
\begin{equation}
\label{F:boundgenusCH}
g(V)
 \leq {\pi}(r,n,d). 
\end{equation}
for the geometric genus of $V\subset \mathbb P^{n+r-1}$. 
In particular  $g(V)=0$ if $d < r(n-1)+2$.
\end{thm}

 An irreducible variety $V\subset\p^{n+r-1}$  as above and  such that $g(V)
 = {\pi}(r,n,d)>0$ is called a {\it Castelnuovo variety}. Note that in this case necessarily  $d \geq  r(n-1)+2$.
 
 \begin{rem}{\rm
 The bound (\ref{F:boundgenusCH}) can be generalized to more general objects than projective  varieties. Indeed, a basic result of web geometry says that  the bound $\rk( W)\leq {\pi}(r,n,d)$ holds for 
the rank $\rk( W)$ of a $r$-codimensional $d$-web $W=W_d(n,r)$ defined on a manifold of dimension $nr$. This result, due to Chern and Griffiths \cite{cherngriffiths}, can be proved by quite elementary methods (see \cite{piriotrepreau} and \cite{pereirapirio}).  Combined with Abel's addition theorem, this  implies  the inequality  $h^0(V,\omega_V)\leq {\pi}(r,n,d)$-- here $\omega_V$ denotes the sheaf of abelian  differential $r$-forms on $V$, see \cite{kunz,lipman,barlet,henkinpassare}--
as soon as $V$ is such that its generic 0-dimensional linear section is in general position in its span, which  is stronger than (\ref{F:boundgenusCH}). }
\end{rem}
 
The classification of projective curves of maximal genus has been obtained by Castel\-nuovo in 1889. More recently, in \cite{harris},  Harris proved the following result.

\begin{prop} Let $V\subset\p^{n+r-1}$ be a Castelnuovo variety of dimension $r\geq 1$ and codimension  at least  $2$. The linear system $| \mathcal I_V(2) |$ cuts out a variety of minimal degree $Y\subset\p^{n+r-1}$ of dimension $r+1$.
\end{prop}

Thus a Castelnuovo variety $V\subset\p^{n+r-1}$ of dimension $r$ is a divisor in the variety of minimal degree $Y\subset\p^{n+r-1}$ cut out by $| \mathcal I_V(2) |$. This property  was used by  Harris to describe
 Castelnuovo varieties (see also the refinements by Ciliberto in \cite{ciliberto}): if $p:\tilde{Y}\rightarrow Y$ denotes a desingularization (obtained for instance by blowing-up the vertex of the cone $Y$ when it is singular), Harris determines the class $[\tilde{V}]$ of $\tilde{V}$ (the strict transform of $V$ in $\tilde{Y}$ via $p$) in the Picard group of $\tilde{Y}$. Assuming (to simplify) that $\tilde{V}$ is smooth, let  $L_V=K_{\tilde{Y}}+\tilde{V}$.
By adjunction theory,    there is a short exact  sequence of sheaves 
$$0 \rightarrow K_{\tilde{Y}} \rightarrow K_{\tilde{Y}}(\tilde{V}) \rightarrow K_{\tilde{V}} \rightarrow 0.$$

 Since $h^0( \tilde{Y},K_{\tilde{Y}}   )=h^1(\tilde{Y},K_{\tilde{Y}})=0$ (because $\tilde{Y}$ is smooth and rational),  the map 
\begin{equation}
H^0(\tilde{Y},L_V)\rightarrow H^0(\tilde{V},K_{\tilde{V}})
\end{equation}
is an isomorphism. Thus it induces  rational maps $   \phi_{V}=  \phi_{|K_V|}$ and  $\Phi_{V}=\Phi_{|L_{V}|} \circ p^{-1}$ such that the following diagram of rational maps is commutative: 
\begin{equation}
\label{F:diagMODEL}
  \xymatrix@R=0.5cm@C=0.7cm{ \;\; {V}  \;    \ar@{^{(}->}[d]  \ar@{-->}[rr]^{\phi_{V}\quad } && \;{\mathbb P}^{\pi(r,n,d)-1} \ar@{=}[d] \\
Y \;      \ar@{-->}[rr]^{\Phi_{V}\quad } && \; {\mathbb P}^{\pi(r,n,d)-1} \; .
} 
\end{equation}
 
Let $X_V$ be  (the closure of)  the image of $\Phi_{V}$. It is an irreducible non-degenerate subvariety in  ${\mathbb P}^{\pi(r,n,d)-1}$ and $\dim(X_V)=\dim(Y)=r+1$. Moreover, one proves that the image by $\Phi_V$ of  a generic 1-dimensional     linear section of $Y$, that is of  a rational normal curve of degree $n-1$ passing through $n$ general points of  $Y$, is a rational normal curve of degree $\dd=d-r(n-1)-2$  contained in $X_V$.  
Thus $X_V\subset\p^{\pi(r,n,d)-1}
=\p^{\overline{\pi}(r,n,\dd)-1}$ is an example of $\overline{X}(r+1,n,\dd)$.
These examples will be called  {\it models associated to the Castelnuovo variety $V\subset\p^{n+r-1}$} 
and also  {\it models of Castelnuovo type}. 
 They are described in detail in \cite{PT} where the authors also prove the following result.
 \begin{thm}
 Let $X\in \overline{X}(r+1,n,\delta)$. If $\delta\neq 2n-3$ then $X$ is of Castelnuovo type.
 \end{thm}
 
It follows that varieties $\overline{X}(r+1,n,\delta)$ not of Castelnuovo type can exists only for $\delta=2n-3$. 
In the sequel, we shall focus on the case $n=3$ and  present some examples of varieties $\overline{X}(r+1,3,3)$ not of Castelnuovo constructed from Jordan algebras.

\subsection{Models of Jordan type associated to  cubic Jordan algebras} 
\label{S:JordanAlgebra}
Recall that a (complex) {\it Jordan algebra} is a $\mathbb C$-vector space  $\mathbb J$  with a $\mathbb C$-bilinear product  $ \mathbb J \times \mathbb J \rightarrow 
\mathbb J$  verifying 
$$
\qquad 
 x^2 \, (y \,x)= 
  (x^2 \,y) \, x
 \qquad \mbox{ for all\;} 
x,y\in \mathbb J\, . 
$$
We will restrict here to the case of  commutative Jordan algebras of finite dimension admitting a unit,   denoted by $e$.  
A classical result in this theory  ensures that a Jordan algebra is  power-associative: for every $x\in \mathbb J$ and any $k\in \mathbb N$, the $k$th-power $x^k$ of $x$ is well-defined. 
This allows us to define   the {\it rank}  of  $ \mathbb J$, denoted by $\rk(\mathbb J)$: it is the dimension (as a complex vector space) of the subalgebra $\mathbb  C[ x]= {\rm Span}_{\scriptstyle{\mathbb C}}\langle x^k\, | {\,  k\in \mathbb N} \rangle$ generated by a general element $x\in \mathbb J$.  Let $m=\rk(\mathbb J)$ and let $\ell_x$ be the restriction of the multiplication by $x$ to  $\mathbb C[ x ]$ and let   $M_x$ be  its associated  minimal polynomial.  Then the relation $M_x(\ell_x)e=\mathbf 0$  expands to
\begin{equation}
\label{polymini}
 x^m-\sigma_1(x)\,x^{m-1}+\cdots+(-1)^m\, \sigma_m(x)\,e=\mathbf 0, 
\end{equation}
where $x\mapsto \sigma_i(x)$ is a homogeneous polynomial map of degree $i$ on $\mathbb J$  (for every $i=1,\ldots,m$). 
By definition, (\ref{polymini}) is the {\it generic minimum polynomial} of the Jordan algebra $\mathbb J$ ({\it cf.} \cite[Proposition II.2.1]{FK}).  Its {\it generic trace} is the linear map $T:x\mapsto T(x)=\sigma_1(x)$ and the homogeneous polynomial map  $N:x\mapsto N(x)=\sigma_m(x)$ of degree $m$ is its {\it generic norm}.  The latter is {\it multiplicative} in the following sense: for all $y\in \mathbb J $,  we have 
$N(x\,x')=N(x)\,N(x')$ 
for every $x,x'\in \mathbb C[y]$ (see \cite[Proposition II.2.2]{FK}).

 One defines the {\it adjoint} $x^\#$ of an element $x\in \mathbb J$ by setting 
  $$x^\#=\sum_{i=0}^{m-1} \sigma_{i}(x) (-x)^{m-1-i}$$
  where $\sigma_0=1$. It follows from 
(\ref{polymini}) that  $x \,x^\#=x^\#x=N(x)\,e$ so that  $N(x)\neq 0$ implies that  $x$ is invertible (for the Jordan product) with inverse $x^{-1}={N(x)}^{-1}x^\#$.
\smallskip 
 
The map $i:x\mapsto x^{-1}$ is a birational involution on $\mathbb J$. 
Let $\Str(\mathbb J)$ be the set of $g\in GL(\mathbb J)$ such that 
$g\circ i\equiv i\circ h$ (as birational maps on $\mathbb J$) for a certain 
$h\in GL(\mathbb J)$. When it exists, such a $h$ is unique: by definition, it is the {\it adjoint} of $g$ and is denoted by $g^\#$. One proves (see \cite{springer}) that $\Str(\mathbb J)$ is a closed algebraic subgroup of $GL(\mathbb J)$, called the {\it structural group} of the  Jordan algebra $\mathbb J$. Moreover, $g\mapsto g^\#$ is an automorphism of $\Str(\mathbb J)$ and there exists a character $\eta:\Str(\mathbb J)\rightarrow \mathbb C^*$ such that 
$ N(g(x))=\eta_g\,N(x) $ for all $x\in \mathbb J$ ({\it cf.} \cite[Prop.1.5]{springer}).  \smallskip

The complex vector space  $\mathcal Z_2(\mathbb J)$ of {\it Zorn's matrices}  is defined by
$$ \mathcal Z_2(\mathbb J)= \bigg\{    
 \begin{pmatrix}
  s & x \\y & t
 \end{pmatrix}\; \Big| 
\begin{tabular}{c}
 $ s,t \in \mathbb C $\\
$x,y \in \mathbb J$
\end{tabular}
\bigg\} . $$

 Assuming from now on that $\mathbb J$ is of rank 3, one defines the  {\it twisted cubic  associated to}  $\mathbb J$  (noted by $X^3_\mathbb J$) as the (Zariski)-closure in $\mathbb P \, \mathcal Z_2(\mathbb J)$ of the image of the polynomial affine embedding 
\begin{align*}
 \label{D:mapcubic}
\nu_3\, : \; 
&\mathbb J   \longrightarrow  \; \mathbb P \, \mathcal Z_2(\mathbb J)
 \\
&x   \longmapsto 
\begin{bmatrix}
  1 & x \\ \,x^\# & N(x)
 \end{bmatrix}.
\end{align*}
Then  the Zariski-closure of ${ \nu_3( \mathbb C\,e)}$ in $\mathbb P \, \mathcal Z_2(\mathbb J)$  is a rational normal curve of degree 3  included in $X^3_\mathbb J$ that we denote by 
$C_{\mathbb J}$:
\begin{equation*}
 C_{\mathbb J}=\bigg\{ 
\begin{bmatrix}
  1 & te  \\ 
t^2e & t^3
\end{bmatrix} \, \bigg| \, t\in \mathbb C \bigg\}
\bigcup 
\bigg\{ \begin{bmatrix}
  0 & 0  \\ 0 &1
\end{bmatrix}
 \bigg\}\,.
\end{equation*}
It contains the following three points 
\begin{equation*}
 {0}_{\mathbb J}=\nu_3(0)=\begin{bmatrix}
  1 & 0 \\ 0 & 0
 \end{bmatrix}\,, \; \quad 
{1}_{\mathbb J}=\nu_3(e)=\begin{bmatrix}
  1 & e \\ e & 1
 \end{bmatrix}\quad 
\mbox{ and }\quad 
{\bf \infty}_{\mathbb J}=\begin{bmatrix}
  0 & 0 \\ 0 & 1
 \end{bmatrix}.
\end{equation*}
\medskip
 
 
Recall that  $i$ stands for  the inverse map $x\mapsto x^{-1}$. From now on we shall suppose  $g\in \Str(\mathbb J)$.   For a fixed $\omega\in \mathbb J$, let $t_\omega: \mathbb J \rightarrow \mathbb J$ be  the translation $x\mapsto x+\omega$ and denote by  $\#$  the bilinear map associated to   $x^\#$: one has 
 $ x\#y=(x+y)^\#-x^\#-y^\#$ for $x,y\in \mathbb J$.

For  $\omega \in \mathbb J$ and $g\in {\rm Str}(\mathbb J)$,  one sets for
every $M= \big[ \!\!\!
\begin{tabular}{c}
 $\, s\; x$ \vspace{-0.1cm}\\
$y\; t$
\end{tabular} \!\!\! \big]
\in \mathbb P \, \mathcal Z_2(\mathbb J)$: 
\begin{align*}
I\big( M ) = & 
\begin{bmatrix}
  t & y \\
  x & s
 \end{bmatrix}, \\  
G_g(M)=&
\begin{bmatrix}
 s & g(x) \\
 \eta_g g^{\#}(y) & \eta_g\,t
 \end{bmatrix}  \\ 
\mbox{and}
 \quad T_\omega(M)= & 
\begin{bmatrix}
  s & x+s\,\omega  \\
  y+ \omega\#x+s\,\omega^\# & t+T(y\,\omega)+T(x\,\omega^\#)+s\,N(\omega) 
 \end{bmatrix}.
\end{align*}
 These maps  are projective automorphisms of $ \mathbb P \, \mathcal Z_2(\mathbb J)$.
\smallskip

 The proof of  the following lemma is straightforward and left to the reader. 
\begin{lemma}
\label{L:automC3J}
For every $\omega \in \mathbb J$ and for every $g\in \Str(\mathbb J)$, we have  
\begin{equation}
 \nu_3\circ  i= I\circ \nu_3\,  , \quad 
\nu_3\circ  g=G_g  \circ \nu_3 \quad 
\mbox{ and } \quad 
\nu_3\circ  t_\omega= T_\omega \circ \nu_3\, .
\end{equation}
Consequently, the maps $I$, $G$ and $T_\omega$ are projective automorphisms of the  cubic $X^3_\mathbb J$. 
\end{lemma}
 
Let $\Conf({\mathbb J})$ be the {\it conformal group} of $\mathbb J$, that is the subgroup of $PGL( \mathcal Z_2(\mathbb J))$ generated by $I$ and the maps $G_g$ and $T_\omega$   for all $\omega \in \mathbb J$ and all $g\in \Str(\mathbb J)$.  From  Lemma \ref{L:automC3J}, it follows that  $\Conf({\mathbb J})$ is a subgroup of the group of projective automorphisms of $X^3_\mathbb J$.
 
\begin{prop}
\label{P:3uplettransitif}
 The group of projective automorphisms of $X^3_\mathbb J$ acts transitively on 3-uple's  of general points of $X^3_\mathbb J$. In fact, if $x_1,x_2,x_3\in \mathbb J$ are sufficiently general, it exists $\mu \in \Conf({\mathbb J}) $ such that $\mu(\nu_3(x_1))=0_{\mathbb J}$, $\mu(\nu_3(x_2) )=1_{\mathbb J}$ and $\mu(\nu_3(x_3))=\infty_{\mathbb J}$. 
\end{prop}
 
\begin{proof} First let $\mathscr G={\rm Str}(\mathbb J)\cdot e$ be the orbit in $\mathbb J$ of the unit $e$  under the linear action of 
the structural group.
By \cite[Theorem 6.5]{springer},  one knows that $\mathscr G$ is a Zariski-open subset of $\mathbb J$. 
 
Assume that $x,y,z\in \mathbb J$ are such that (1) $y_1=y-x$ and $z_1=z-x$ are  invertible; (2) $z_2=z_1^{-1}-y_1^{-1}\in \mathscr G$, {\it i.e} $z_2=g(e)$ for a certain $g\in \Str(\mathbb J)$. Then set 
$$\mu=
(G_g)^{-1}
\circ
T_ {
-(y_1)^{-1}} \circ  I\circ T_ {
-x} \in \Conf({\mathbb J}) . $$
We leave to  the reader to verify that $\mu(\nu_3(x))=\infty_{\mathbb J}$, $\mu(\nu_3(y))=0_{\mathbb J}$ and  $\mu(\nu_3(z))=1_{\mathbb J}$. Since $\nu_3(\mathbb J)$ is dense in $X^3_\mathbb J$,  the conclusion
follows.
\end{proof}
 
Let $x_1,x_2,x_3 \in \mathbb  J$ and  let $\mu\in \Conf({\mathbb J})$ 
be as in the statement of Proposition \ref{P:3uplettransitif}. Since $\mu\in PGL(\mathcal Z_2(\mathbb J))$, the curve $\mu^{-1}(C_{\mathbb J})$ is a rational normal curve of degree 3 passing through the points $\nu_3(x_i)$ for $i=1,2,3$. Since $\mu$ is a projective  automorphism of $X^3_\mathbb J$ (by Lemma \ref{L:automC3J} above), this twisted cubic curve is also contained in $X^3_\mathbb J$. Thus  we have  proved the
following result.

\begin{coro}
 The twisted cubic $X^3_\mathbb J$ associated to a Jordan algebra $\mathbb J$ of rank three is 3-covered by rational normal curves of degree 3. 
\end{coro}
 
Let $k$ be the dimension of a Jordan algebra $\mathbb J$ of rank 3. Then $X^3_\mathbb J$ is a non-degenerate algebraic subvariety of the projective space $ \mathbb P  \mathcal Z_2(\mathbb J)$, whose dimension
is $2k+1$. Since $\pi(k,3,3)=2k+2$, one obtains  that $X^3_\mathbb J\subset \mathbb P  \mathcal Z_2(\mathbb J)=\p^{2k+1}$ is an example  of $\overline{X}(k,3,3)$.  Thus the cubics $X^3_\mathbb J$  associated to Jordan algebras of rank 3  are  examples of varieties of type $\overline{X}(k,3,3)$, which will be called  {\it  models of Jordan type}. 
 A consequence of Theorem \ref{Cremonascroll} is that a  $X^3_\mathbb J$ 
 is never of Castelnuovo type.  \medskip
 
 In fact, using Theorem \ref{Cremonascroll}, one can easily prove the following criterion which characterizes  the varieties $\overline{X}(r+1,n,2n-3)$ of Castelnuovo type for arbitrary $n\geq 3$.
 
 \begin{prop} A variety $X=\overline{X}(r+1,n,2n-3)$ is of Castelnuovo type 
 if and only if  for any $n-2$ general point $x_2,\ldots,x_{n-1}\in X$, the intersection
 of the space 
 $\oplus_{i=2}^{n-1}Osc_X^1(x_i)$ with $X$ contains a hypersurface.
 \end{prop} 
 
In fact, it can be verified that the $X=\overline{X}(r+1,n,2n-3)$  of Castelnuovo type  are exactly the two  scrolls
 $S_{n-2,\ldots,n-2,n}$ and $S_{n-2,\ldots,n-2,n-1,n-1}$ in $\mathbb P^{(n-1)(r+2)-1}$.

\subsubsection{{\rm {\bf Examples}}}
\label{S:examplesJ=C+W}
 We shall  describe some explicit examples of 
models of  Jordan type. \smallskip 
 
A Jordan algebra $\mathbb J$ is {\it simple} if it does not admit proper non-trivial ideals. It is {\it semi-simple} if it is a finite direct sum of simple proper ideals (or equivalently, if the bilinear symmetric form $(x,y)\mapsto T(xy)$ is non-degenerate, see \cite[Theorem 5 p. 240]{jacobson}). 
 
 \begin{ex}
\label{exJ2}
{\rm Let $B$ be a  symmetric  bilinear form on $W=\mathbb C^{r-1}$. Then $\mathbb J'=\mathbb C\oplus W$ with the product defined by 
\begin{equation}
 \label{Jordanproduct}
(\lambda,y)\diamond (\lambda',y')=\big(\lambda\lambda'-B( y,y') , \lambda\,y'+\lambda'\,y\big)
\end{equation}
is a  Jordan algebra of rank 2. 
Moreover,
 $x^2-2\lambda\,x+(\lambda^2+B(y,y))\,e=0$ for any $x=(\lambda,y)\in \mathbb J'$ (where $e=(1,0)$). The generic norm and the generic trace are  $N(\lambda,y)=\lambda^2+B(y,y)$ and  $T(\lambda,y)=2\lambda$ hence the adjoint is given by $x^\#=(\lambda,-y)$. 
 One verifies that  $\mathbb C\oplus W$ is semi-simple if and only  if $B$ is non-degenerate and  $\mathbb C\oplus W$  is simple if $B$ is non-degenerate and $r>2$. When $r=2$, $\mathbb J'=\mathbb C\oplus W$ with the product (\ref{Jordanproduct}) is isomorphic to the direct product (of Jordan algebras) $\mathbb C\times \mathbb C$.}
 \end{ex}

One can define the  {\it conic associated to a Jordan algebra}  $\mathbb J'$ of rank 2. By definition, it is the (Zariski)-closure,  denoted by $X^2_{\mathbb J'}$, of the image of the  affine map $ \mathbb J' \ni x \mapsto [1:x:N(x)   ] \in 
\mathbb P( \mathbb C \oplus \mathbb J' \oplus \mathbb C )$. Since $N$ is homogeneous of degree 2, $X^2_{\mathbb J'}$ is a non-degenerate quadric hypersurface in $\mathbb P^{r+1}$ (where $r={\rm dim}_{\mathbb C}\, \mathbb J'$). It is smooth if and only if $\mathbb J'$ is simple.  In any case, $X^2_{\mathbb J'}$ is an example of $\overline{X}(r,3,2)$ (of minimal type).

\begin{lemma}
\label{L:tyu}
If $\mathbb J=\mathbb C\times \mathbb J'$ is the direct product of $\mathbb C$ with a Jordan algebra $\mathbb J'$ of rank 2 and dimension $r$, then $\mathbb J$ is of rank 3 and $
X^3_\mathbb J \simeq \mathbb P^1\times X^2_{\mathbb J'}$
 Segre embedded in $\mathbb P^{2r+3}$.
 \end{lemma}

This result ensures the existence of examples of models of Jordan type of any dimension.  For  $\mathbb J=\mathbb C\times\mathbb J'$ with $\mathbb J'$ of rank 2, note that  $X^3_\mathbb J$  is smooth if and only if $\mathbb J$ is semi-simple. 
 
\subsubsection{Examples of $\overline{X}(r+1,3,3)$'s associated to simple Jordan algebras of rank 3}

Let us now consider models of Jordan type associated to simple  Jordan algebras of rank 3, which  can be completely classified. We recall the 
well known Hurwitz  Theorem  that there are exactly four composition  algebras over the field of real numbers: $\mathbb R$ itself, the field $\mathbb C$ of complex numbers and the algebras  $\mathbb H$ and  $\mathbb O$ of quaternions and octonions respectively (the reader can consult \cite{baez} for  a nice  introduction to these objects). 
 
If $\mathbb A$ denotes one of these algebras, let $\mathbb A_{\mathbb C}=\mathbb A \otimes \mathbb C$ be its complexification (over $\mathbb R$). Any such $\mathbb A_{\mathbb C}$ is a {\it complex composition algebra}: for $x_i=a_i \otimes r_i\in\mathbb A_\mathbb C$, with $a_i\in \mathbb A$ and $r_i \in \mathbb C$ for $i=1,2$, we define $x_1\cdot x_2=a_1a_2 \otimes r_1r_2$, $\overline{x_1}=\overline{a_1}\otimes {r_1}$ and $\Vert   x_1\Vert^2=x_1\cdot \overline{x_1}\in \mathbb C$. Of course,  except for $\mathbb O_{\mathbb C}$, there are classical isomorphisms (of complex algebras) 
\begin{equation*}
 \mathbb R_{\mathbb C}\simeq \mathbb C\, , \quad 
 \mathbb C_{\mathbb C}\simeq \mathbb C\oplus \mathbb C \quad \mbox{ and } \quad 
\mathbb H_{\mathbb C}\simeq M_{2}(\mathbb C). 
\end{equation*}
 
Let $\mathscr H_3(\mathbb A_{\mathbb C})$ be
the space of Hermitian matrices of order three
with coefficients in $\mathbb A_{\mathbb C}$:
\begin{equation*}
  \mathscr H_3(\mathbb A_{\mathbb C})= \Bigg\{    
 \begin{pmatrix}
  r_1 & \overline{x_3} &  \overline{x_2} \\x_3 & r_2 &   \overline{x_1} 
\\ x_2& x_1& r_3
 \end{pmatrix}\; \Bigg| 
\begin{tabular}{c}
 $ x_1,x_2,x_3 \in \mathbb A_{\mathbb C} $\\
$r_1,r_2,r_3 \in \mathbb C$
\end{tabular}
\Bigg\} .
\end{equation*}
 
The multiplication  
\begin{equation}\label{E:jordanproduct}
 (M, N)\mapsto  \frac{1}{2} ( MN+NM)
\end{equation}
induces on  $\mathscr H_3(\mathbb A_{\mathbb C})$ a structure of complex (unital commutative) Jordan algebra. For $\mathbb A=\mathbb R,\mathbb C$ or $\mathbb H$, it is  a direct consequence of the fact that $\mathbb A$ and hence the rings of $3\times 3$ matrices,  $M_3(\mathbb A)$, are associative algebras. For $M_3(\mathbb O)$ a particular argument is needed  and we refer to \cite[Chap. V and VIII]{FK} for  this case.  One proves (see again \cite{FK}) that any Jordan algebra  $ \mathscr H_3(\mathbb A_{\mathbb C})$ is simple and of rank 3. 

 We are  now able to describe all  simple Jordan algebras of rank three. Their classification is classical ({\it cf.} \cite[p.233]{jacobson} or \cite{FK} for instance) and  is given in Table 1  below with a description of  the corresponding cubic curves, see also  \cite{mukai,landsbergmanivel,clerc}. The table also shows   classical isomorphisms of $ \mathscr H_3(\mathbb A_{\mathbb C})$ with some matrix algebras having  \eqref{E:jordanproduct} as Jordan product (in the case $\mathbb A=\mathbb R,\mathbb C$ or $\mathbb H$).
\begin{table}[h]\label{tableJordan}

  \centering
  \begin{tabular}{|l|c|}
  \hline
 {\bf Jordan algebra $\mathbb J$}    &  {\bf Twisted cubic curve $X^3_\mathbb J$ over $\mathbb J$ } \\
 \hline     
$ \mathscr H_3(\mathbb R_{\mathbb C})\simeq {\rm Sym}_3(\mathbb C)$  
 &     
\begin{tabular}{l} 6-dimensional Lagrangian  \\
grassmannian  $LG_3(\mathbb C^6)\subset \mathbb P^{13}$
\end{tabular}
\\
  \hline
$\mathscr H_3(\mathbb C_{\mathbb C})  \simeq M_3(\mathbb C)$  
 &     
\begin{tabular}{l} 9-dimensional Grassmannian \\
manifold $G_3(\mathbb C^6)\subset \mathbb P^{19}$
\end{tabular}
\\
  \hline
$ \mathscr H_3(\mathbb H_{\mathbb C})\simeq{\rm Alt}_6(\mathbb C)$ 
 &     
\begin{tabular}{l} 15-dimensional orthogonal  \\
Grassmannian  $OG_6(\mathbb C^{12})\subset \mathbb P^{31}$
\end{tabular}
\\
  \hline
$ 
 \mathscr H_3(\mathbb O_{\mathbb C})
$   & 
\begin{tabular}{l} 27-dimensional $E_7$-variety  
in   $ \mathbb P^{55}$
\end{tabular}
\\
  \hline
\end{tabular}
\vspace{0.15cm}
\label{T:jordan}
\caption{{\small 
Simple (unital and finite-dimensional) Jordan algebras of rank 3 and their associated cubic curves.}
}
\end{table}\\
It follows from \cite{landsbergmanivel,clerc} that the cubic curves $X^3_{\mathbb J}$ associated to one of the  simple Jordan algebras presented in Table 1  are homogeneous varieties, yielding non-singular  examples of $\overline{X}(k,3,3)$ of Jordan type, for $k=6,9,15$ and $27$.  
 
 The four varieties of the last column in Table 1 have been studied by several authors from many points of view. The interested reader can consult for example \cite{mukai, landsbergmanivel0, landsbergmanivel, landsbergmanivel3, clerc}.

\subsubsection{A Jordan cubic curve associated to the sextonions} 
\label{S:sextonions}
 The algebra of (complex) sextonions  $\mathbb S_{\mathbb C}$ is an alternative   algebra over $\mathbb C$ such that $\mathbb H_{\mathbb C}\subset \mathbb S_{\mathbb C} \subset \mathbb O_{\mathbb C}$. It is of (complex) dimension 6 and has been constructed in \cite{landsbergmanivel2,westbury}. 
 
The product (\ref{E:jordanproduct}) realizes  $ \mathscr H_3(\mathbb S_{\mathbb C})$ as a 21-dimensional sub-Jordan algebra of $ \mathscr H_3(\mathbb O_{\mathbb C})$. Then $ \mathscr H_3(\mathbb S_{\mathbb C})$ is of rank 3 but is not semi-simple ({\it cf.} \cite[\S 8.2]{landsbergmanivel2}).  The cubic curve over $ \mathscr H_3(\mathbb S_{\mathbb C})$ is denoted by $G_\omega(\mathbb S^3,\mathbb S^6)$ in \cite{landsbergmanivel2} where it is explained  why it can be considered as a kind of  Lagrangian Grassmannian. It is a  $\overline{X}(21, 3,3)$ in $\mathbb P^{43}$ that is quasi-homogeneous and singular along a quadric of dimension 10 ({\it cf.} \cite[Corollary 8.14]{landsbergmanivel2}).

\subsection{Some cubic curves associated to  associative algebras} \label{S:associativeALGEBRA}
Let $A$ be an associative algebra (of finite dimension) not necessarily commutative but with a unit $e$.  
Let $A^+$ denote the algebra $A$ endowed with the product $x*y=\frac{1}{2}(xy+yx)$. If $A$ is commutative,  then $A=A^+$. It is immediate to see  that in general $A^+$ is a  Jordan algebra with $e$ as unit.  
We will say that a Jordan algebra is {\it special } if it is isomorphic to a subalgebra of a Jordan algebra of the form $A^+$ with $A$ associative. For instance, the simple Jordan algebras $ \mathscr H_3(\mathbb A_{\mathbb C})$  in Table 1 are special except when $\mathbb A=\mathbb O$.

Being associative,  $A$  is also power-associative so that  one can define its  {\it rank} as introduced 
 at the beginning of \S \ref{S:JordanAlgebra}. Of course, the rank of the associative algebra $A$ and the rank  of the associated Jordan algebra $A^+$ coincide. 
\medskip 
 
Complex associative algebras  have been classified in low dimensions --e.g. in dimension 3, 4 and 5-- in classical or more recent papers and looking at these lists one immediately  computes the rank  of these  algebras. Then, considering the $A^+$'s associated to rank three associative  $A$'s, one obtains examples of
 special cubic Jordan algebras in dimension 3, 4 and 5.  The computations needed to obtain the examples  appearing in the following  subsections are elementary but tedious and quite long so that they  will not be reproduced here. 
 
\subsubsection{The $\overline{X}(3,3,3)$'s associated to special Jordan 
 algebras of dimension three} The classification of  3-dimensional complex associative algebras
is classical, see \cite{study,scorzaALGEBRA3}. We refer to   \cite{FPP,gabriel} for  more recent references. 
\begin{thm}
\label{T:associativeALGEBRAdim3}
 A 3-dimensional complex  associative (unital)  algebra of rank 3 is isomorphic to one of the following ones:
\begin{align*}
 A_1= \mathbb C \times \mathbb C\times \mathbb C\; , \qquad 
A_2=\mathbb C\times \frac{\mathbb C[X]}{\,(X^2)} \; , \qquad 
A_3= \frac{\mathbb C[X]}{\;(X^3)}
\, .
\end{align*} 
\end{thm}
The above algebras  are commutative so that  $A_i^+=A_i$ (for $i=1,2,3$)  are examples of special 3-dimensional cubic Jordan algebras.  Let $X_{A_i}$ be the cubic curve associated to $ A_i^+$. Clearly, the cubic curve $X_{A_1}$ is nothing but Segre's threefold  $\mathbb P^1\times \mathbb P^1\times \mathbb P^1\subset \mathbb P^7$, which  is a particular case of the general construction of Lemma \ref{L:tyu}. Similarly, 
one verifies  that $X_{A_2}$  is isomorphic to the Segre embedding  of $\mathbb P^1\times S_{02}$ in $\mathbb P^7$. 
\smallskip 
 
On the other hand the cubic curve associated to $A_3$ yields a new example.  One takes  $x=1, y=X$ and $z=X^2$ as a $\mathbb C$-basis of $A_3$. Since this algebra  is commutative, the Jordan product coincides with the associative one of $A_3$ and   may be expressed as follows: 
 $$(x,y,z)\cdot(x',y',z')=(xx',xy'+x'y,xz'+x'z+yy')\; . $$
Then  the generic norm and the  adjoint of $(x,y,z)\in A_3$ are 
given by the following formulae:
$$  (x,y,z)^\#=(x^2,-xy,y^2-xz) \qquad \mbox{and} \qquad N(x,y,z)=x^3  \,.  $$ 

So the cubic curve  $X_{A_3}$  is the closure of the image of the rational map 
\begin{equation}
\label{E:paramdiamond3} 
[x:y:z:t]\longmapsto \big[t^3:xt^2:yt^2:zt^2:x^2t:-xyt:(y^2-xz)t:x^3  \big].\end{equation}
 
The variety $X_{A_3}\subset\p^7$  can also be described as the image of $\mathbb P^3$  by the rational map associated to the linear system of cubic surfaces  passing through three infinitely near double points
(two of them generate the line $x=t=0$ contained in the base locus scheme of the linear system) so that it has degree 6.
 
  \subsubsection{Some $\overline{X}(4,3,3)$ associated to Jordan 
algebras of dimension four} 
\label{S:associativeALGEBRA4}
The classification of  4-dimensional complex associative algebras
is also classical, see \cite{study,scorzaALGEBRA4}, and  was also reconsidered more recently 
in  \cite{gabriel}, see also the references in these papers.  As explained above, we are essentially  interested in those of rank 3, whose classification is contained in the next result 
where we use the  labels and notation  of \cite[p.\,151-152]{gabriel}.

\begin{prop}
\label{T:associativeALGEBRAdim4}
Let $A$ be a   4-dimensional associative algebra of rank three. Then if $A$ is commutative, it is isomorphic to one of the following algebras
$$
A_6= \mathbb C \times \frac{\mathbb C[X,Y]\;}{\,(X,Y)^2} \; , \qquad A_7=\frac{\mathbb C[X,Y]}{\,(X^2,Y^2)}
\; , \qquad A_8= \frac{\mathbb C[X,Y]}{\,(X^3,XY,Y^2)}\;
$$

If $A$ is not commutative,  then one of the following holds: 
\begin{itemize}
\item   $A$  is isomorphic to one of the following three triangular matrix algebras
$$
A_{13}=\mathbb C\times 
\begin{pmatrix}
\mathbb C & \mathbb C \\
0 & \mathbb C \\
\end{pmatrix} ; \;
\qquad  
A_{14}=
\left\{
\begin{pmatrix}
a & 0 & 0 \\
c & a & 0 \\
d&0&b
\end{pmatrix}\; 
\bigg\lvert \; a,b,c,d\in \mathbb C\right\}  ; \; \qquad  
A_{15}=
(A_{14})^{opp}   \; ;
$$
\item  there exists $\lambda \in \mathbb C\setminus\{1\}$   such that $A$ is isomorphic to 
$$ A_{18(\lambda)}=\frac{\mathbb C\langle X,Y\rangle }{\,(X^2,Y^2,YX-\lambda XY)}\;;  $$

\item  $A$ is isomorphic to 
$$
A_{19}=\frac{\mathbb C\langle X,Y\rangle }{\,(Y^2,X^2+YX, XY+YX)}\,.
$$
\end{itemize}
(Here $\mathbb C\langle X,Y\rangle$ stands for the free associative  algebra generated by 1, $X$ and $Y$). 
\end{prop}
\medskip

Since $A_{15}=(A_{14})^{opp}$ one has $A_{14}^+= A_{15}^+$. Moreover, one verifies easily (via elementary  computations) that the Jordan algebras $A_{18(\lambda)}^+$ and   $A_{19}^+$ are associative so that they  are isomorphic to one of the three commutative associative algebras $A_6,A_7$ or $A_8$ (in fact one has $A_{18(\lambda)}^+\simeq A_7$ and $A_{19}^+\simeq A_8$). Thus we get the following result.
\medskip

\begin{coro}
\label{C:JordanDIM4}
Let $\mathbb J$ be a cubic Jordan algebra of dimension 4 of the form $A^+$ with $A$ associative and of rank three. Then $\mathbb J$ is isomorphic to one of the following algebras: 
\begin{equation}
\label{E:JordanDIM4}
A_{6}\;,\qquad 
A_{7}\;,\qquad 
A_{8}\;,\qquad 
A_{13}^+\;\qquad  \mbox{or} \qquad 
A_{14}^+.
\end{equation}
\end{coro}

The reader has to be aware that not every 4-dimensional cubic Jordan algebra is of the form $A^+$ with $A$ associative. For instance (as explained above), if $\mathbb J'$ stands for the simple Jordan algebra of rank 2 on $\mathbb C^3$ then the  direct product $ \mathbb C \times \mathbb J'$ is a non-associative Jordan cubic algebra not isomorphic to any of the Jordan algebras in   (\ref{E:JordanDIM4}). In this case, the associated cubic curve $X_{\mathbb C \times \mathbb J'}$
is $\mathbb P^1\times Q\subset \mathbb P^9$ where $Q$ is a smooth hyperquadric in $\mathbb P^4$. 

Another example is given by the Jordan algebra denoted by $\mathbb J_* $, the Jordan product of which  is explicitly given by 
$$
x*y=\Big( -x_1y_1\, , \, x_2y_2 \, , \,     x_4y_4-x_1y_3-x_3y_1\, , \, 
\frac{1}{2}\big(
x_2y_4+x_4y_2-x_1y_4-x_4y_1
\big)
\Big)
$$
for $x=(x_i)_{i=1}^4$ and $y=(y_i)_{i=1}^4$ in $\mathbb C^4=\mathbb J_*$.  One verifies that  $\mathbb J_*$  is indeed of rank 3.
\medskip

After some easy computations, one obtains explicit affine parametrizations of the form $x\mapsto [1:x:x^\#:N(x)]$
of the cubic curves associated to the  cubic Jordan algebras  mentioned above in this subsection.
We collect them in the following table for further reference. 
\begin{table}[h]
\label{Table:jordanDIM4}
\centering
\begin{tabular}{|c|l|c|c |}\hline
{\bf Algebra $\mathbb J$} &  \hspace{1.75cm}   {\bf Adjoint}  $\boldsymbol{ x^\# }$   & {\bf  Norm} $\boldsymbol{N(x)}$   & $X_{\mathbb J}\subset \mathbb P^9$  \\
\hline
$A_{6}^+=A_{6}$   & $\big({x_2}^2\,,\,x_1x_2\,, -x_1x_3\,,-x_1x_4\big)$   &    $x_1\,{x_2}^2$   &$\mathbb P^1\times S_{002}$  \\ 
\hline
$A_{7}^+=A_{7}$  & $
\big({x_1}^2\,, - x_1x_2\,, -x_1x_3\,,\, 2x_2x_3-x_1x_4\big) 
$   &    ${x_1}^3$& \\ 
\hline
$A_{8}^+=A_{8}$  & $
\big({x_1}^2\,, - x_1x_2\,, -x_1x_3\,, \,{x_2}^2-x_1x_4\big) 
$   &    ${x_1}^3$& \\ 
\hline
$A_{13}^+\qquad\;$  & $(x_2\,x_4\,,\, x_1x_4\,, -x_1x_3\,,\,x_1x_2)$   &    $x_1\,x_2\,x_4$& $\mathbb P^1\times S_{011}$     \\ 
\hline
$A_{14}^+\qquad\;$  & $(x_1x_2,{x_1}^2,-x_2x_3,-x_1x_4)$   &    ${x_1}^2\,x_2$&      \\ 
\hline$\mathbb C\times \mathbb J'$ \hspace{0.24cm}  & $
(x_2^2+x_3^2+x_4^2\,,\, x_1x_2 \,,\,  -x_1x_3 \,,\,  -x_1x_4)$   &    $x_1(x_2^2+x_3^2+x_4^2)$&   $\mathbb P^1\times Q$     \\ 
\hline$\mathbb J_{*}$ \hspace{0.8cm} & $(x_1x_2\, , \,{ x_1}^2\, , \,   {x_4}^2-x_2x_3 \, , \,  x_1x_4)$   &    ${x_1}^2\, x_2$&      \\ 
\hline
\end{tabular}
\caption{ }
\end{table}
\medskip

\subsection{Some other examples of $\overline{X}(r+1,n,2n-3)$ when $n>3$}
\label{S:otherexamples}
According to the main result of \cite{PT}, a variety $\overline{X}(r+1,n,\dd)$ is of Castelnuovo type except maybe when $n>2$, $r>1$ and $\dd=2n-3$. 
The cubic curves associated to Jordan algebras of rank three provide examples of $\overline{X}(r+1,3,3)$ that are not of Castelnuovo type. It is  natural to try to produce  some examples of $\overline{X}(r+1,n,2n-3)$ not of Castelnuovo type for $n>3$. 
\smallskip
 
We are aware essentially only of  examples which are closely related to varieties 3-covered by twisted cubics. 

The Veronese manifold $v_3(\mathbb P^3)\subset \mathbb P^{19}$ is of course a 
$\overline{X}(3,2,3)$  but  is also a 
$\overline{X}(3,6,9)$.  Indeed, since $\mathbb P^3$ is $6$-covered by twisted cubics, it follows that $v_3(\mathbb P^3)$ is 6-covered by rational normal curves of degree 9. Of course, $v_3(\mathbb P^3)$ is not of Castelnuovo type (since $\deg(v_3(\mathbb P^3))=27$ whereas the degree of a Castelnuovo model $X=\overline{X}(3,6,9)$ is 17). Now let $x_1,x_2,x_3$ be three generic points on $v_3(\mathbb P^3)$. For $I\subset \{1,2,3\}$ of cardinality $i\leq 3$, let $X_I$ be the image of $v_3(\mathbb P^3)$ by the osculating projection of center $S_I$ defined as the span of the 1-osculating spaces of 
$v_3(\mathbb P^3)$ at the points $x_i$ with  $i\in I$.  Then $X_I$ is non-degenerate in $\mathbb P^{19-4i}$, is $(6-i)$-covered by rational normal curves of degree $9-2i$ hence is an example of $\overline{X}(3,6-i,9-2i)$. When $i=3$, one has  $X_I=\mathbb P^1\times \mathbb P^1\times\mathbb P^1$ hence this example is not new. But for $i=1$ and $i=2$, one obtains respectively two new examples of $\overline{X}(3,5,7)$ and $\overline{X}(3,4,5)$  that are not of Castelnuovo type. \medskip


\section{Classification of  projective varieties 3-covered by twisted cubics}
\label{S:3-covered}

By definition $\overline{\pi}(r,3,3)=2r+4$ for every $r\geq 1$. 
 In this Section we shall classify  varieties
 $X=\overline X(r+1,3,3)\subset\p^{2r+3}$ for $r$ small and/or under suitable hypothesis. 

Let us recall some facts, which were proved in the previous sections or which are easy consequences of them.
\medskip

\begin{lemma}\label{prel} Let $X=\overline X(r+1,3,3)\subset\p^{2r+3}$. Then:
\begin{enumerate}
\item\label{birational} the tangential projection $\pi_T:X\map\p^{r+1}$ from the tangent space $T=T_x X=Osc^1_X(x)$ at a general point $x\in X$ is birational. In particular $X$ is a rational variety,
$SX=\p^{2r+3}$ and $X$ is not a cone.
\item\label{linnorm} The variety $X\subset\p^{2r+3}$ is not the birational projection from an external point of a variety $X'\subset\p^{2r+4}$.
\end{enumerate}
\end{lemma}
\begin{proof} The family of twisted cubics passing through $x$ is 2-covering and a general
twisted cubic in this family projects from $T$ onto a general line in $\p^{r+1}$ so that the birationality of $\pi_T$ follows from Lemma \ref{L:Siena} and the first part is proved.

Suppose that $X=\pi_p(X')$ with $p\in\p^{2r+4}\setminus X'$. Let $x_i\in X$, $i=1,2,3$, be general points and let $x'_i\in X'$ such that $\pi_p(x'_i)=x_i$ for every $i=1,2,3$.
If $C\subset X$ is the unique twisted cubic passing through $x_1,x_2, x_3$ and if $C'\subset X'$ is its strict transform on $X'$, then $C'$ is a rational curve passing through
$x'_1, x'_2, x'_3$ such that $\pi_p(C')=C$, yielding $\deg(C')=\deg(C)=3$. This would imply  $X'=X(r+1,3,3)\subset\p^{2r+4}$, and we would obtain $2r+5\leq\overline\pi(r,3,3)=2r+4.$
This contradiction concludes the proof.
\end{proof}
\medskip

By definition, there exists an irreducible family of twisted cubics, let us say $\Sigma$, contained in $X=\overline X(r+1, 3,3)\subset\p^{2r+3}$. Moreover, $\Sigma$ has dimension $3r$, is 3-covering and  through three
general points of $X$ there passes a unique twisted cubic belonging to $\Sigma$. The family of twisted cubics in $\Sigma$ passing through a general point $x\in X$ 
contains an irreducible component of dimension $2r$ which is 2-covering for  $X$. These twisted cubics are mapped by $\pi_T$ onto the lines in $\p^{r+1}$
and a general line in $\p^{r+1}$ is  the image via  $\pi_T$ of a twisted cubic passing through $x$.
The birational map
$$\phi=\pi_T^{-1}:\p^{r+1}\map  X\subset\p^{2r+3}$$
is thus given by a linear system of cubic hypersurfaces mapping a general line of $\p^{r+1}$ birationally onto a twisted cubic passing through $x$. The general cubic
hypersurface in this linear system is mapped by $\phi$ birationally onto a general hyperplane section of $X$.

Let $\overline \pi_x:\overline X=\Bl_x (X)\to X$ be the blow-up of $X$ at $x$ and let $E=\p^r$ be the exceptional
divisor of $\overline \pi_x$. Let $\overline\pi_T=\overline\pi_x\circ\pi_T:\overline X\map \p^{r+1}$. The  restriction of $\overline\pi_T$ to $E$ is  defined by a linear
system  $|II_{X,x}|$ of quadric hypersurfaces in $E=\p^r$, the so called {\it second fundamental form of $X$ at $x$}. We shall denote by $B_x\subset E=\p^r$ the {\it base locus scheme of $|II_{X,x}|$}.

Since $X\subset\p^{2r+3}$ is non-degenerate, the birational
map $\overline\pi_T$ is defined at the general point of $E$. We claim that  $E'=\overline \pi_{T|E}(E)=\p^r\subset\p^{r+1}$ is a hyperplane and that 
 the restriction of $\overline \pi_T$ to $E$ is birational onto its image.
Indeed, if $\dim(E')<r$, then a general line in $\p^{r+1}$ would not cut $E'$ and its image by $\phi$ would not pass
through $x$. If $\deg(E')\geq 2$, then a general line $l\subset\p^{r+1}$ would cut $E'$ at $\deg(E')$ distinct points where $\phi$ is defined. From $\phi(E')=x$ we would deduce that  $\phi(l)$ is singular at $x$, in contrast with the fact that $\phi(l)$ is a twisted cubic. From this picture
it also immediately follows the birationality of the restriction of $\pi_T$ to $E$.

Therefore $\dim(|II_{X,x}|)=r$ and $\overline\pi_{T|E}:E\map E'$ is a Cremona transformation not defined along $B_x$, the base locus scheme
of $|II_{X,x}|$. Moreover, since $\phi(E')=x$, the restriction
of the linear system of cubic hypersurfaces defining $\phi$ to $E'$ is constant and given by a cubic hypersurface $\mathscr C'_x\subset E'=\p^r$. One can assume that  $E'\subset\p^{r+1}$ is cut out by  $x_0=0$. Let $\boldsymbol  x=(x_1,\ldots,x_{r+1})\in \mathbb C^{r+1}$
and let $f(\boldsymbol x)$ be a cubic equation for $\mathscr C'_x\subset E'$.
Let us choose homogeneous coordinates $(y_0:\cdots:y_{2r+3})$ on $\p^{2r+3}$ such that $x=(0:\cdots:0:1)$ and $T_xX=V(y_0,\ldots,y_{r+1})$. 

The map $\phi:\p^{r+1}\map X\subset\p^{2r+3}$ is given by $2r+4$ cubic polynomials
$g_0,\ldots, g_{2r+3}$. We can suppose that $x_0$ does not divide $g_{2r+3}$ and that $x_0$ divides $g_j$ for every $j=0,\ldots, 2r+2$. Moreover  $x_0^2$ divides $g_0,\ldots g_{r+1}$ since the hyperplane sections of the form $\lambda_0g_0+\cdots+\lambda_{r+1}g_{r+1}=0$ correspond
to hyperplane sections of $X$ containing $T_xX$ and hence having at least a double point at $x$. Modulo a change of coordinates on $\p^{r+1}$ we can thus suppose $g_i=x_0^2x_i$ for every $i=0,\ldots, r+1$ and that $g_{2r+3}=x_0g+f$, with $g=g(\boldsymbol x)$ quadratic polynomial. The hyperplane sections
of $X$ passing through $x$ and not containing $T_xX$ are smooth at $x$ from which it follows that we can also suppose $g_{r+2+j}=x_0f_j$ with $j=0,\ldots, r$ and $f_j=f_j(\boldsymbol x)$ quadratic polynomials. By Lemma \ref{prel} we can also suppose $g=0$, or equivalently $g\in\langle f_0,\ldots, f_r\rangle$. Otherwise
$X\subset\p^{2r+3}$ would be the birational projection on the hyperplane $V(y_{2r+4})=\p^{2r+3}\subset\p^{2r+4}$ from  the external point $(0:\ldots:0:1:-1)\in\p^{2r+4}$
of the variety $X'\subset\p^{2r+4}$ having the parametrization $\widetilde \phi:\p^{r+1}\map X'\subset\p^{2r+4}$ given by the following homogenous cubic polynomials: $\widetilde g_i=g_i$ for $i=0,\ldots, 2r+2$; $\widetilde g_{2r+3}=x_0g$ and $\widetilde g_{2r+4}=f$.

By blowing-up the point $x$ on $\p^{2r+3}$ it immediately follows that  the restriction of $\overline\pi_T^{-1}$
to $E'$ is given by $(f_0:\cdots:f_r)$. Hence $\psi_x:=\overline\pi_{T|E}:E\map E'$ is either a projective transformation or 
a Cremona transformation of type $(2,2)$, i.e. given by quadratic forms without a common factor and such that the inverse is also given by quadratic forms without a common factor. 
In conclusion we can suppose  that  the rational map $\phi$ is given by the $2r+4$ cubic polynomials  
\begin{equation}\label{pargeneral}
x_0^3, x_0^2x_1,\ldots, x_0^2x_{r+1},x_0f_0,\ldots,x_0f_r, f
\end{equation}
 and that the base locus of $\psi_x^{-1}$, $B'_x\subset\p^r=E'$,
is $V(f_0,\ldots, f_r)\subset\p^r$, where in this case the polynomials $f_i(\boldsymbol x)$ are considered as polynomials in the variables $x_1,\ldots, x_{r+1}$.

 Let $t_xX$ denote  the affine tangent space to $X$ at $x$.  The first principal result in this section is the following.

\begin{thm}\label{Cremonascroll} Let $X=\overline X(r+1,3,3)\subset\p^{2r+3}$ and let notation be as above. Let $x\in X$ be general and let $\psi_x:\p^r\map\p^r$ be the associated Cremona transformation.
Then the following conditions are equivalent:
\begin{enumerate}
\item[(a)] $\psi_x$ is equivalent  to a projective transformation as a birational map;
\item[(b)]  $X$ is projectively equivalent either to $S_{1\ldots122}$ or to $S_{1\ldots113}$;
\item[(c)] the affine parametrization deduced from \eqref{pargeneral} is, respectively,
 either
  $$\big(1: x_1:\ldots: x_{r+1}: x_1^2: x_1x_2:\ldots :x_1x_{r+1}:x_1^2x_2\big)$$ 
  or
$$\big(1: x_1:\ldots: x_{r+1}: x_1^2,x_1x_2:\ldots :x_1x_{r+1}:x_1^3\big);$$
\item[(d)]  the projection from $T=T_xX$ of a general
twisted cubic  included in $X$ is a conic. 
\end{enumerate}
\medskip

If $X=\overline X(r+1,3,3)\subset\p^{2r+3}$ is not a rational normal scroll as above, then:
\begin{enumerate}
\item the associated Cremona transformation $\psi_x$ is of type $(2,2)$;
\item the  linear system defining $\phi:\p^{r+1}\map X\subset\p^{2r+3}$ 
consists of  the cubic hypersurfaces in $\p^{r+1}$ having double points along $B'_x\subset E'\subset\p^{r+1}$;
\item  the scheme $B_x\subset\p^{r}$ is equal (as scheme) to   $\mathcal L_x\subset\p^{r}$, the Hilbert scheme of lines contained in $X$ and passing through $x$ in its natural embedding into $E=\p((t_xX)^*)$. Moreover, $B'_x\subset E'=\p^r$ and $B_x=\mathcal L_x\subset E=\p^r$ 
are projectively equivalent  so that $\psi_x$ and its inverse have the same base loci, modulo this identification.
\item if $X$ is also smooth, then $B_x=\mathcal L_x$ and $B'_x$ are smooth schemes.
\end{enumerate}
\end{thm}
\begin{proof} The birational map $\psi_x$ is equivalent to a projective transformation as a birational map if and only if the linear system $|II_{X,x}|$ has a fixed component, which is necessarily a hyperplane.
It is well known, see also \cite[(3.21)]{GH}, that this happens if and only if $X\subset\p^{2r+3}$ is a scroll in the classical sense. From Lemma \ref{prel} we deduce that this happens if and only if $X\subset\p^{2r+3}$
is a smooth rational normal scroll, yielding the equivalence of (a) and (b). If $X\subset\p^{2r+3}$ is a rational normal scroll as before, a general twisted cubic $\overline C\subset X$ 
cuts $\p^r_x=T_xX\cap X$ in a point so that $\pi_T(\overline C)=\widetilde C\subset\p^{r+1}$ is a conic cutting $E'$ in two points, possibly coincident. These two points are contained in the base locus
of $\psi$ and one of this point is  double for the general cubic hypersurface in the linear system defining $\psi$ since $\psi(\widetilde C)=\overline C$ is a twisted cubic. From this we deduce
that $\mathscr C'_x$ consists of a double hyperplane $\Pi\subset E'$, which is exactly $B'_x$, and of another hyperplane, possibly infinitely near to $\Pi$. By reversing the construction we see that if $\pi_T(\overline C)=\widetilde C$ is a conic, then $\mathscr C'_x$ is as before and the general cubic hypersurface defining $\phi$ has a double point along a hyperplane $\Pi\subset E'$, which is easily seen to be equal to $B'_x$, yielding that $X\subset\p^{2r+3}$ is a scroll. Therefore also  the equivalence of (b) and (d) is proved. If we suppose that $B'_x=\p^{r-1}\subset E'$ is given by $x_0=x_1=0$ and if we
take into account the previous description we immediately deduce the equivalence between (c) and (d).

Suppose from now on that $X=\overline X(r+1, 3,3)\subset\p^{2r+3}$ is not a rational normal scroll so that by the previous
part the associated Cremona map $\psi_x$ is of type $(2,2)$, proving (1).

By the discussion above on $\p^{r+1}\setminus E'$ the map $\phi$ has an affine expression 
$$\phi(\boldsymbol x)=\big(1:x_1:\cdots:x_{r+1}:f_0(\boldsymbol x):\cdots:f_r(\boldsymbol x):f(\boldsymbol x)\big).$$ 

Let $(y_0:\dots:y_{2r+3})$ be a system of homogeneous coordinates on $\p^{2r+3}$ as above.
Then the equations defining $X\subset\p^{2r+3}$ in the affine space $\mathbb A^{2r+3}=\p^{2r+3}\setminus V(y_0)$ are $y_i=x_i$, $i=1,\ldots, r+1$; $y_{r+2+j}=f_j(\boldsymbol x)$, $j=0,\ldots, r$ and  $y_{2r+3}=f(\boldsymbol x)$, that is, 
letting $\boldsymbol y=(y_1,\ldots, y_{r+1})$, we get the equations $y_{r+2+j}=f_j(\boldsymbol y)$  for  $j=0,\ldots, r$ and  $y_{2r+3}=f(\boldsymbol y)$.

 Let 
$p= \phi(1:\boldsymbol p)=  (1:\boldsymbol p:f_0(\boldsymbol p):\cdots:f_r(\boldsymbol p):f(\boldsymbol p))$ be a general point of $X$, with  $\boldsymbol p=(p_1,\ldots, p_{r+1}) \in  \mathbb C^{r+1}$. 
In particular  $(0:\boldsymbol p)$ is a general point on $E'$. A  tangent direction at $p\in X$ corresponds to the image via $d\phi_q$ of the tangent direction 
to some  line passing through $q=(1:\boldsymbol p)\in\p^{r+1}$. We shall parametrize lines through $q$ via points $(0:\boldsymbol y)\in E'$
so that such a line, denoted by $L_{\boldsymbol y}$, admits 
 $t\mapsto \boldsymbol p+ t\boldsymbol y$ as an affine parametrization. 
Then for $i=0,\dots,r$, one has 
\begin{equation}\label{parquad}
f_i(\boldsymbol p+t\boldsymbol y)=f_i(\boldsymbol p)+2t f^1_i(\boldsymbol p,\boldsymbol y)+t^2f_i(\boldsymbol y),
\end{equation}
where  $f^1_i$  stands for the bilinear form associated to the quadratic form $f_i$. Moreover
\begin{equation}\label{parquad2}
f(\boldsymbol p+t\boldsymbol y)=f(\boldsymbol p)+t f(\boldsymbol p,\boldsymbol y)+t^2
f(\boldsymbol y,\boldsymbol p)+t^3f(\boldsymbol y),
\end{equation}
where $f(\boldsymbol p,\boldsymbol y)=df_{\boldsymbol p}(\boldsymbol y)$ is quadratic in $\boldsymbol p$ and linear in $\boldsymbol y$.

Clearly,  the base locus of the
second fundamental form at $p=\phi(1:\boldsymbol p)$ is the scheme 
$$B_p=V\big(f_0(\boldsymbol y),\ldots ,  f_r(\boldsymbol y), f(\boldsymbol y,\boldsymbol p)\big)=V\big(f_0(\boldsymbol y),\ldots ,  f_r(\boldsymbol y)\big)\subset\p\big((t_pX)^*\big),$$
where the second equality of schemes follows from the equality $\dim(<f_0,\ldots,f_r>)=r+1$ combined with 
the fact that $\dim(|II_{X,p}|)=\dim(|II_{X,x}|)$ by the generality of $p\in X$.
In particular we deduce that for 
$\boldsymbol z\in B_p$ we have  
$ f( {\boldsymbol  z},\boldsymbol p)=0$.  Because $(0:\boldsymbol p)$ is 
general in $ E'$, this  implies  $df_{{\boldsymbol  z}}=0$ (since  $f( {\boldsymbol  z},\boldsymbol p)=df_{{\boldsymbol  z}}(\boldsymbol p)$  for every $\boldsymbol p$) on one hand, 
and gives $f(\boldsymbol z)=0$ on the other hand (since $0=f(\boldsymbol z,\boldsymbol z)=3f(\boldsymbol z)$ after specializing  $\boldsymbol p=\boldsymbol z$). The previous facts show that 
the cubic $\mathscr C'_x=V(f(\boldsymbol x))\subset\p^{r+1}$ has double points along $B'_x$ and part (2) is proved. 
From these facts it immediately
follows also that the closure of the image of the line $L_{{\boldsymbol z}}$ (for $\boldsymbol z\in B_p$) via the map $\phi$ is a line included in $X$ and passing through $p$, proving (3).
Put more intrinsically, the equation of $\mathcal L_p$, the Hilbert scheme of lines contained in $X$ and passing
through $p$ in its natural embedding into $\p((t_pX)^*)$, is the scheme defined by the equations
$f_j(\boldsymbol x), f(\boldsymbol x, \boldsymbol p), f(\boldsymbol x)$ and we proved that the ideal generated by these
polynomials coincides with the ideal generated by the $f_j$'s which defines $B_p$ as a scheme.

To prove (4) we recall that for  a smooth variety $X\subset\p^N$  the scheme $\mathcal L_x\subset\p((t_x X)^*)$, when non-empty, is a smooth scheme for $x \in X$ general, see for example \cite[Proposition 2.2]{QEL1}.
\end{proof}
\medskip

Let $$\varphi=(\varphi_0, \ldots , \varphi_r): \mathbb P^r\dashrightarrow \mathbb P^r$$ be a Cremona transformation of bidegree $(2,2)$. Let  $B$, respectively $B'$, be the base locus  of $\varphi$, respectively of $\varphi^{-1}$. The classification of such maps is known for $r=2$, $r=3$ (see \cite{PRV}) and for $r=4$ in the generic case (see \cite{Semple}).
From this classification one deduces that in low dimension the base loci $B$
and $B'$ are projectively equivalent so that, modulo a projective transformation, every Cremona transformation of bidegree $(2,2)$ in $\p^r$, $r\leq 4$, is an involution. As a consequence of Theorem \ref{Cremonascroll} we deduce below that this holds for arbitrary $r\geq 2$ a priori and not as a consequence of the classification. As far we know, this question has not been addressed in the literature.
\medskip 

\begin{coro}\label{involutions} 
Let $\varphi=(\varphi_0, \ldots , \varphi_r): \mathbb P^r\dashrightarrow \mathbb P^r$ be a Cremona transformation of bidegree $(2,2)$ with $r\geq 2$. Let  $B$, respectively $B'$, be the base locus  of $\varphi$, respectively of $\varphi^{-1}$. Then $B$ and $B'$ are projectively equivalent.
\end{coro}
\begin{proof}
Consider  $\pi_1:\Bl_{B}(\p^r)\to \p^r$, the blow-up of $\p^r$ along $B$
and  $\pi_2:\Bl_{B'}(\p^r)\to \p^r$, the blow-up of $\p^r$ along 
$B'$. We deduce  the following
diagram of birational maps:
\begin{equation}\label{diagram}
\xymatrix{
          &  \Bl_{B}(\p^r)=\Bl_{B'}(\p^r)\ar[dl]_{\pi_1}\ar[dr]^{\pi_2}    \subset\p^r\times\p^r     &            \\
  \p^r  \ar@{-->}[rr]_{\varphi} & &\p^r              }
\end{equation}
where $\pi_i$ are naturally identified with the restriction of the projections on each factor.
The equality $\Bl_{B}(\p^r)=\Bl_{B'}(\p^r)$, from now on indicated with $\widetilde\p^r$,  follows from the fact that these reduced and irreducible schemes
are the closure of the graph of the maps $\varphi$ and $\varphi^{-1}$ inside $\p^r\times\p^r$.

Let $E_1=\pi_1^{-1}(B)$ and $E_2=\pi_2^{-1}(B')$ be the $\pi_i$-exceptional Cartier divisors, $i=1,2$, defined by the invertible sheaves $\pi_1^{-1}\mathcal I_{B}\cdot\O_{\widetilde\p^r}$, respectively $\pi_2^{-1}\mathcal I_{B'}\cdot\O_{\widetilde\p^r}$.
Let $H_i\in
|\pi_i^*(\mathcal O(1))|$ for $i=1,2$.  We have
\begin{equation}\label{firstrel}
H_2\sim 2H_1-E_1\qquad\mbox{and}\qquad  H_1\sim 2H_2-E_2
\end{equation}
from which   we deduce  
\begin{equation}\label{secondrel}
E_1\sim 3H_2-2E_2\qquad\mbox{and}\qquad  E_2\sim 3H_1-2E_1.
\end{equation}

Let ${\boldsymbol x}=(x_1:\ldots:x_{r+1})$ be homogeneous coordinates in $\p^r$, which we shall consider as the hyperplane $V(x_0)$ on $\p^{r+1}$ with homogeneous coordinates $(x_0:x_1:\cdots:x_{r+1})$. 
Let $C_1=\pi_1(E_2)=V(n({\boldsymbol x}))\subset\p^r$ and let $C_2=\pi_2(E_1)=V(\widetilde 
n({\boldsymbol  x}))$.
 By the above description   $C_1\subset\p^r$ is a cubic hypersurface 
 singular  along $B$,
 that is the partial derivatives of $n({\boldsymbol x})$ also belong to the homogeneous saturated ideal 
 $I_{B}\subset\mathbb C[x_1,\ldots,x_{r+1}]$. One also immediately proves that $C_1$ is the so called {\it secant scheme of $B$}, that is the scheme defined by the image of the universal family of  lines in $\p^{2r+3}$ over the  lines generated by length 2 subschemes of $B$. From this one deduces another proof that $C_1$ is singular along $B$.
 
The map $\phi:\p^{r+1}\map\p^{2r+3}$ given by
$$\phi(x_0,{\boldsymbol x}) =\big(x_0^3:x_0^2x_1:\cdots:x_0^2x_{r+1}:x_0\varphi_0({\boldsymbol x}):\cdots:x_0\varphi_r({\boldsymbol x}):n({\boldsymbol x})\big)$$
is birational onto the closure of its image $X\subset\p^{2r+3}$. We claim that $X=\overline X(r+1,3,3)$ so that the conclusion will follow from part (3) of Theorem \ref{Cremonascroll}. Indeed let
$p_1,p_2,p_3\in\p^{r+1}$ be three general points, let $\Pi=\langle p_1,p_2,p_3\rangle$ be the plane they span and let $L=V(x_0)\cap \Pi$. Then $L\subset V(x_0)=\p^r$ is a general line so that $\varphi(L)=C$ is a conic in $\p^r$ cutting $B'$ in a length
three scheme $R'$ spanning a plane $\Pi'$. Then $\varphi^{-1}({\Pi'})=\Pi$ is a plane containing $L$ and cutting $B$ in a length three  scheme $R$ spanning $\Pi$. Then $<p_1,p_2,p_3,R>=\p^3$ and through the length six scheme
$p_1\cup p_2\cup p_3\cup R$ there passes a unique twisted cubic $\widetilde C$. Then $\phi(\widetilde C)$ is a twisted cubic since the linear systems defining $\phi$ consists of cubic hypersurfaces having double points  along $B$ and it passes through the three general points $\phi(p_1),\phi(p_2),\phi(p_3)\in X$.
\end{proof} 
\medskip

Let the notation be as above.
Then,  modulo composition to the left by a linear map, one can assume that $B=B'$. This implies in particular  that there exist $\ell \in {\rm End}(\mathbb C^{r+1})$ 
such that $\varphi^{-1}=\ell\circ\varphi$. Thus  there  exists   a cubic form $n({\boldsymbol x})$  such that 
\begin{equation}\label{invol}
\ell \circ \varphi\circ  \varphi ({\boldsymbol x})=n({\boldsymbol x}){\boldsymbol x}
\end{equation}
for every ${\boldsymbol x}=(x_1,\ldots,x_{r+1})\in \mathbb C^{r+1}$. \medskip 
  
In substance, Theorem \ref{Cremonascroll} and the construction in 
Corollary \ref{involutions} say that every  $\overline{X}(r+1,3,3)$ not of Castelnuovo type 
determines (and is determined by) a quadro-quadric Cremona transformation.  
 We point out 
that the previous remark has the following interesting geometric consequence.
\medskip

\begin{coro}\label{oadp} Let $X=\overline X(r+1,3,3)\subset\p^{2r+3}$. Then $X$ is   a variety with one apparent double point, that is there passes a unique secant line to $X$ 
 through a general  point $q\in\p^{2r+3}$.
\end{coro}
\begin{proof} Let notation be as above and let $q=(1:{\boldsymbol q}:{\boldsymbol q'}: z)\in\p^{2r+3}$
be a fixed general point, where ${\boldsymbol  q}, {\boldsymbol  q'}\in\mathbb C^{r+1}$ and $ z\in\mathbb C^*$ by generality
of $q$. Moreover, we can also suppose that all the pairs of distinct points $p_1,p_2\in X$ such that $q\in\langle p_1,p_2\rangle$ are of the form $p_i=(1:{\boldsymbol  x_i}:{\boldsymbol  x_i'}:n({\boldsymbol  x_i}))$ with ${\boldsymbol  x_i}\in \mathbb C^{r+1}$ for  $i=1,2$.
We shall essentially  argue as  in \cite[Proposition 8.4]{clerc} (see also \cite[Proposition 8.16]{landsbergmanivel2}).
Modulo a translation in $\mathbb C^{r+1}$ we can suppose without loss of generality that ${\boldsymbol q}=\mathbf 0$, that 
$n({\boldsymbol q'})\neq 0$ and that 
$p_i=(1:{\boldsymbol  x_i}:\varphi({\boldsymbol  x_i}):n({\boldsymbol x_i}))\in X$, $i=1,2$, are two distinct points such that $q\in\langle p_1,p_2\rangle$. We have to show that  the equation 
$$q=\lambda p_1+\mu p_2$$
has  uniquely determined solutions.
The above equation  splits into the following four: 
\begin{equation}
\label{E:tatayoyo}
\lambda+\mu=1, \qquad {\boldsymbol x_2}=-\frac{\lambda}{\mu}{\boldsymbol  x_1}, \quad \varphi({\boldsymbol x_1})=\frac{\mu}{\lambda}{\boldsymbol   q'}\quad \mbox{and}\quad n(
{\boldsymbol  x_1})=\frac{\mu^2}{\lambda(\mu-\lambda)} z.
\end{equation}

Thus relation \eqref{invol} implies  
$$n({\boldsymbol x_1}){\boldsymbol x_1}=(\ell\circ \varphi)(\varphi({\boldsymbol x_1}))=\frac{\mu^2}{\lambda^2}(\ell\circ\varphi)({\boldsymbol q'})$$
so that 
$${\boldsymbol x_1}=\frac{\mu-\lambda}{\lambda z}(\ell\circ\varphi)({\boldsymbol    q'})\qquad
\mbox{and}\qquad
 {\boldsymbol x_2}=\frac{\lambda-\mu}{\mu z}(\ell\circ\varphi)({\boldsymbol   q'}).$$

We deduce  that $p_1,p_2\in X$ are uniquely determined by $q$ as soon as we show that $(\lambda,\mu)$ are uniquely determined by this point.
Since $\lambda+\mu=1$, we shall prove that $\lambda\mu$ is uniquely determined. 

From relation \eqref{invol}, it follows that there exists a cubic form $m$ such that 
$ \varphi\circ  \ell \circ\varphi ({\boldsymbol y})=m({\boldsymbol y}){\boldsymbol y}
$  for every ${\boldsymbol y}\in \mathbb C^{r+1}$. This yields  $m(\varphi({\boldsymbol x}))=n({\boldsymbol x})^2$  for every ${\boldsymbol x}$.  Applying $m$
to $\varphi({\boldsymbol  x_1})=\frac{\mu}{\lambda}{\boldsymbol   q'}$ we deduce
$n({\boldsymbol  x_1})^2=\frac{\mu^3}{\lambda^3}m({\boldsymbol q'})$.  Combining this with 
the last relation in (\ref{E:tatayoyo}) and remarking that  $(\mu-\lambda)^2=1-4\lambda\mu$ (since $\lambda+\mu=1$), we finally get 
$$\lambda\mu=\frac{m({\boldsymbol q'})}{ z^2+4m({\boldsymbol  q'})}\; ,$$
concluding
the proof.
\end{proof}
\smallskip

\begin{rem}{\rm Following a classical approach of Bronowski, see \cite{Bronowski},  it was proved in \cite{cmr} that
an irreducible variety $X\subset\p^{2r+3}$ with one apparent double point (OADP variety) 
 projects birationally onto $\p^{r+1}$ from a general tangent space $T=T_xX$,
see also \cite{cilibertorusso} for generalizations.
Letting notation be as in the discussion before Theorem \ref{Cremonascroll}, then $\pi_T(E)=E'\subset\p^{r+1}$
is hypersurface of degree $d=\deg(E')\geq 1$, which is a birational projection of the quadratic Veronese embedding
of $\p^r$. In particular $1\leq d\leq 2^r$.  In
 \cite[Theorem 5.3]{cr2} it is proved that for normal OADP varieties having $d=\deg(E')=1$, not scrolls over a curve, the birational map $\pi_T^{-1}:\p^{r+1}\map X\subset\p^{2r+3}$ is given
 by a linear system of cubics hypersurfaces having double points along the base locus of the quadro-quadric Cremona transformation $\pi_{T|E}^{-1}:E'\map E$. In particular this class of normal OADP varieties is contained in the class
 of $X=\overline X(r+1,3,3)\subset\p^{2r+3}$ so that the subsequent classification results of $\overline X(r+1,3,3)$'s of different
 kind or dimension can be reformulated for the above class of normal varieties, see \cite{cr2}. Conjecturally every $\overline X(r+1,3,3)$ should be projectively
 equivalent to a $X_{\mathbb J}\subset\p^{2r+3}$, see discussion in Remark \ref{classjordangeom} below. The known examples  of twisted cubics over cubic comples Jordan algebras are  normal varieties even if we are not aware of a general proof of this fact and neither of the fact that a $\overline X(r+1,3,3)$
 is a priori normal.
 If these were true,  one would deduce a  one to one correspondence between normal OADP varieties with $d=\deg(E')=1$ and $\overline X(r+1,3,3)$ and also probably with twisted cubic over Jordan algebras.}
\end{rem}
\medskip

Cremona transformations have been studied since a long time and several classification results have been obtained, especially for quadro-quadric transformations. These classifications can be used to describe all the  $\overline{X}(r+1,3,3)$  in low dimension or under suitable hypothesis. We shall begin by recalling some easy results on Cremona transformations of type $(2,2)$ having smooth base loci, see also \cite{EinShB} and \cite[\S 4]{QEL1} for the study of  related  objects. 
\medskip

\begin{prop}\label{special22} Let $\varphi:\p^r\map\p^r$ be a Cremona transformation of type $(2,2)$ whose base locus $B\subset\p^r$ is smooth. Then one of the following holds:
\begin{enumerate}
\item $r\geq 2$, $B=Q^{r-2}\amalg p$ with $Q^{r-2}$ a smooth quadric hypersurface and $p\not\in\langle Q^{r-2}\rangle$;
\item $r=5$ and $B$ is projectively equivalent to the Veronese surface $\nu_2(\p^2)$;
\item $r=8$ and $B$ is projectively equivalent to the Segre variety $\p^2\times\p^2$;
\item $r=14$ and $B$ is projectively equivalent to the Grassmann variety $G_2(\mathbb C^6)$;
\item $r=26$ and $B$ is projectively equivalent to the 16-dimensional $E_6$ variety.
\end{enumerate}
\end{prop}
\begin{proof} Let notation be as in the proof of Corollary \ref{involutions}. Let $B_1,\ldots, B_s$ be the irreducible components of $B$ and let $B'_1,\ldots, B'_t$ be the irreducible components of $B'$. It is easy to see that the general quadric hypersurface defining $\varphi$ is smooth at every point of $B$.

The smoothness of $B$ ensures that $B_j\cap B_l=\emptyset$
for every $j\neq l$ so that $\widetilde \p^r$ is smooth and naturally isomorphic to the successive blow-up of the $B_i$'s in some
order.  In particular $s=t$ (see also  Corollary \ref{involutions}). 

Let $E_i=\pi_1^{-1}(B_i)$ and let $E'_i=\pi_2^{-1}(B'_i)$. Let $H_j=\pi_j^*(H)$ with $j=1,2$ and with $H\subset\p^r$ a hyperplane.
We have  the following linear equivalence relation of divisors on $\widetilde\p^r$, see proof of Corollary \ref{involutions}:
$$E_1+\cdots+E_s\sim 3H_2-2(E'_1+\cdots+E'_s)$$ and  $$E'_1+\cdots+E'_s\sim 3H_1-2(E_1+\cdots+E_s).$$ Thus the scheme-theoretic images $\pi_2(E_1+\cdots+E_s)$  and
$\pi_1(E'_1+\cdots+E'_s)$ have degree 3, yielding $s\leq 3$. 

Suppose $s=3$. Then $\deg(\pi_2 (E_i))=\deg(\pi_1(E'_i))=1$ for every $i=1,2,3$ so that $B_i$ and $B'_i$ are a linearly embedded $\p^{r-2}$ since the intersection of two distinct hyperplanes  $\pi_1(E'_i)$, respectively $\pi_2(E_i)$, is contained in the base locus.
The smoothness of a general quadric defining $\varphi$ along each $B_i=\p^{r-2}\subset\p^r$ implies
$r-2\leq\frac{r-1}{2}$, that is $r\leq 3$. Thus necessarily $r=2$ since $h^0(\mathcal I_B(2))=r+1$ and we are in case (1).

If $s=2$, we can suppose $\deg(\pi_1(E'_1))=2$ and $\deg(\pi_1(E'_2))=1$. Thus the quadric $\pi_1(E'_1)\cap \pi_1(E'_2)$ is an irreducible component, let us say $B_1$, of $B$. The birationality of $\varphi$ implies $h^0(\I_B(2))=r+1$. Since $h^0(\I_B(2))\leq h^0(\I_{B_1}(2))=r+2$, we see that $B_2$ consists of only one point and  we are in case (1).

If $s=1$, the above diagram \eqref{diagram} shows that for general $q\in \pi_1(E'_1)\setminus B$ there exists a linear space $\p^{r-1-\dim(B')}_q$ passing through $q$ and cutting
$X$ along a quadric hypersurface of dimension $r-2-\dim(B)$. If $\varphi(q)=q'$, then naturally 
$\p^{r-1-\dim(B')}=\p((N_{B'/\p^r})_x^*)$. This immediately implies that $\pi_1(E'_1)$ is the variety of secant lines to $B$ and that $B\subset\p^r$ is a $QEL$-manifold
of type $\delta(B)=\frac{1}{2}\dim(B)$, see \cite[Proposition 4.2]{QEL1}. Indeed, $r-2-\dim(B)=\delta(B)=2\dim(B)+1-\dim(\pi_1(E'_1))$ yields $\dim(B)=\frac{2}{3}(r-2)$, see also the computations in  \cite{EinShB}. Thus $B\subset\p^r$ is a $QEL$-manifold and also a Severi variety in the sense of Zak.  The classification of Severi varieties due to Zak, see \cite{Zak} and  also \cite[Corollary 3.2]{QEL1}, assures that we are in one of the cases (3)--(6).
\end{proof}
\smallskip

The classification of arbitrary  $\overline X(r+1,3,3)\subset\p^{2r+3}$ is difficult due to the existence of a lot of
singular examples. By Theorem \ref{Cremonascroll} and Corollary \ref{involutions} this classification is closely related to that of all Cremona involutions of type $(2,2)$ on $\p^r$ and also to the classification of arbitrary complex cubic Jordan algebras of dimension $r+1$. 
On the contrary for  smooth $\overline X(r+1,3,3)$'s  a complete classification is  possible due to Proposition \ref{special22} and Theorem \ref{Cremonascroll}.
\smallskip

\begin{thm}\label{smooth33} Let $X=\overline X(r+1,3,3)\subset\p^{2r+3}$ be smooth. Then one of the following holds, modulo projective equivalence:
\begin{enumerate}
\item[(i)]  $X$ is either $S_{1\ldots122}$ or $S_{1\ldots113}$;
\item[(ii)] $X$ is the Segre embedding $\p^1\times Q^r\subset\p^{2r+3}$ with $Q^r$ a smooth hyperquadric;
\item[(iii)] $r=5$ and $X$ is the Lagrangian Grassmannian $LG_3(\mathbb C^6)\subset\p^{13}$;
\item[(iv)] $r=8$ and $X$ is  the  Grassmannian $G_3(\mathbb C^6)\subset\p^{19}$;
\item[(v)] $r=14$ and $X$ is  the Orthogonal Grassmannian $OG_6(\mathbb C^{12})\subset\p^{31}$;
\item[(vi)] $r=26$ and $X$ is  the $E_7$--variety in $\p^{55}$.
\end{enumerate}
\end{thm}
\begin{proof} If the associated Cremona transformation is equivalent to a projective transformation we are in case (i) by Theorem \ref{Cremonascroll}.
Otherwise, by Theorem \ref{Cremonascroll},  the associated Cremona transformation $\psi_x$ is of type $(2,2)$ with smooth base locus.
 Let  $\phi:\p^{r+1}\map X\subset\p^{2r+3}$ be the birational representation of $X$ given by the linear system of cubic hypersurfaces having double points
along $B'_x$. Then $B'_x$ is projectively equivalent to a variety as in cases (1)--(5) of Proposition \ref{special22} so that $X$ is
as in cases (ii)--(vi) by a well known description of the corresponding varieties, see  for example \cite{landsbergmanivel0,landsbergmanivel, mukai}.
\end{proof}
\medskip

Now we apply the classification of quadro-quadric Cremona transformations in low dimension to deduce the corresponding classification for varieties 3-covered by twisted cubics. For instance, since every birational map $\varphi: \mathbb P^1\dashrightarrow \mathbb P^1$ is equivalent to  a projective transformation, one immediately deduces that a surface $\overline{X}(2,3,3)\subset \mathbb P^5$  is necessarily of Castelnuovo type, namely it is one of the scrolls $S_{13}$ or $S_{22}$. Then,  by projecting from $m-3$ general points, one gets the following result: 

\begin{coro}\label{class233} Let $X=\overline X(2,m,m)\subset\p^{m+2}$, $m\geq 3$. Then $X$ is projectively equivalent to a smooth rational normal
scroll of degree $m+1$.
\end{coro}

The classification of  birational maps of degree 2  on $\mathbb P^2$ is classical, easy and well known.  From this we shall immediately deduce the classification of arbitrary $\overline X(3,3,3)\subset\p^7$.

\begin{coro}\label{class333} Let $X=\overline X(3,3,3)\subset\p^7$. Then $X$ is projectively equivalent to \begin{enumerate}
 \item a smooth rational normal scroll of degree 5, that is $S_{113}$ or $S_{122}$; or
\item the variety $\p^1\times Q^2\subset\p^7$ 
where $Q^2\subset\p^3$ is  an irreducible quadric; 
 or 
\item the normal del Pezzo 3-fold obtained as the image of $\p^3$ by the linear system of cubic
surfaces having three infinitely near double points or equivalently described as the twisted cubic over the Jordan algebra
$A_3$ of Theorem \ref {T:associativeALGEBRAdim3} and admitting the parametrization (\ref{E:paramdiamond3}).
\end{enumerate}
\end{coro}
\begin{proof} If $\psi_x:\p^2\map\p^2$ is equivalent to a projective transformation, then we are in case (1) by Proposition \ref{Cremonascroll}. Otherwise
$\psi_x:\p^2\map\p^2$ is a  Cremona transformation of type $(2,2)$. If $\psi_x$  is the ordinary quadratic transformation, then the cubic surfaces defining $\phi:\p^3\map X\subset\p^7$
have three double distinct points at the indeterminacy points of $\phi_x^{-1}$ so that this linear system coincides with the complete linear system of cubic surfaces having these three double points. In this case $X\subset\p^7$ is
projectively equivalent to the Segre embedding of $\p^1\times Q^2$ with $Q^2$ a smooth quadric surface, see also Theorem 
\ref{smooth33}.

If $\psi_x$ has two infinitely near base points and another base  point, reasoning as above we have that $X\subset\p^7$ is projectively equivalent to the Segre embedding of $\p^1\times S_{02}$.

If $\psi_x$ has three infinitely near base points, then $\phi$ is given by a  linear system of cubic surfaces having three infinitely near double points and we are in case (3).
\end{proof}
\medskip

Using the last result,  we can now classify irreducible 3-folds $X=\overline X(3,m,m)\subset\p^{m+4}$ for every $m\geq 4$. If $m\geq 4$, by projecting from $m-3$ general points
we get an irreducible 3-fold $X_{m-3}\subset\p^{7}$ which is 3-covered by twisted cubics. If $X_{m-3}$ is a smooth rational normal scroll of degree 5, then
$X\subset\p^{m+4}$ is a smooth rational normal scroll of degree $m+2$. If $X_{m-3}\subset\p^7$ is a  3-fold as in Corollary \ref{class333}, then $X\subset\p^{m+4}$
would be a normal del Pezzo 3-fold of degree $6+m-3=m+3\geq 7$, which is linearly normal. Moreover since $SX_{m-3}=\p^7$, we deduce $\dim(SX)=7$.
The normal del Pezzo 3-folds $X\subset\p^{m+4}$ of degree $m+3$, $m\geq 4$, are smooth and with $\dim(SX)=6$, see \cite{fujita2}. In conclusion
we have proved the following result:

\begin{coro}\label{class3mm} Let $X=\overline X(3, m,m)\subset\p^{m+4}$ with $m\geq 4$. Then  $X$ is  a smooth rational normal scroll of degree $m+2$.
\end{coro}

Also the classification of quadratic Cremona transformation   on $\p^3$  is known. By comparing Tableau 1 in \cite{PRV} and Table 2 above one obtains the following result.

\begin{thm}
\label{22p3} Let $\psi:\p^3\map\p^3$ be a Cremona transformation of bidegree $(2,2)$, not equivalent to a projective transformation. Then 
up to composition to the right and to the left by linear automorphisms, $\psi$ can be assumed to be one of the seven quadratic  involutions $x\mapsto x^\#$ defining the adjoint of one of   the seven cubic Jordan algebras on $\mathbb C^4$ described in Section  \ref{S:associativeALGEBRA4} 
  (see Table 2 for explicit formulae). 
\end{thm}
\medskip

From Theorem \ref{22p3} , it follows immediately the

\begin{coro}\label{class433} Let $X=\overline X(4,3,3)\subset\p^9$. Then $X$ is projectively equivalent to one of the following varieties: 
\begin{enumerate}
 \item a smooth rational normal scroll of degree 6, that is $S_{1113}$ or $S_{1122}$; or
\item a cubic curve over one of the seven Jordan algebras in Table 2.
\end{enumerate}
\end{coro}
\medskip

\begin{rem}\label{classjordangeom}
{\rm

The varieties appearing in Theorem \ref{smooth33} are  also  particular examples
of smooth {\it Legendrian varieties}, see \cite{landsbergmanivel3} for definitions, some related results and references. They are also interesting examples of smooth  varieties with one apparent double
point defined before by Corollary \ref{oadp}, see also \cite{mukai} and \cite{cmr}.

The class of Cremona transformations $\psi:\p^r\map\p^r$ of type $(2,2)$ arising from two different birational projections from one point of an irreducible quadric hypersurface
$Q^r$ will be called {\it elementary quadratic transformations}, that is if $\pi_1:Q^r\map \p^r$ and if $\pi_2:Q^r\map\p^r$ are the two projections we have $\psi=\pi_2\circ\pi_1^{-1}$.
The well known classification of plane quadratic Cremona transformations and the results of Theorem \ref{22p3} say that for $r\leq 3$ every Cremona transformation
of type $(2,2)$ is an elementary quadratic transformation. These examples yield varieties  $\overline X(r+1,3,3)\subset\p^{2r+3}$, not rational normal scrolls, which are either
the Segre embedding of $\p^1\times Q^r$ or projective degenerations of them when some base point becomes infinitely near.

There is an interesting  approach to
Jordan algebras developed by T. A. Springer \cite[\S 1.27]{springer} and based on {\it $j$--structures} and indirectly also on
the so called {\it Hua's identities}, see \cite{mccrimmon}.
These results and our geometrical treatment yield the following consequence, probably well known to the experts: two twisted cubic curves over Jordan algebras
$X_{\mathbb J_i}\subset\p^{2r+3}$, $i=1,2$, are projectively equivalent if and only if  $\mathbb J_1$ and $\mathbb J_2$ are isomorphic Jordan algebras.

In particular in Theorem \ref{smooth33}, Corollary \ref{class333} and Corollary \ref{class433} we obtained geometrical proofs of the classification of all cubic Jordan algebras $\mathbb J$ such that the associated 
twisted cubic $X_{\mathbb J}\subset\p^{2r+3}$ is respectively smooth, of dimension 3, of dimension 4.

Based on the results of Theorem \ref{Cremonascroll}, of the construction in the proof of Corollary \ref{involutions}, of Theorem \ref{smooth33}, of Corollary \ref{class333} and of Corollary \ref{class433}, one could 
ask the following question: 
\begin{quote}
{\it Is a $\overline X(r+1,3,3)\subset\p^{2r+3}$ not of Castelnuovo type projectively equivalent to a twisted cubic $X_{\mathbb J}\subset\p^{2r+3}$
for some cubic Jordan algebra $\mathbb J$ of dimension $r+1$?} 
\end{quote}

We conjecture that the answer to the previous question is affirmative.
In other terms, we conjecture that the following a priori unrelated mathematical objects  in fact coincide: 
\begin{itemize}
\item  varieties $\overline{X}(r+1,3,3)\subset \mathbb P^{2r+3}$, up to projective equivalence;
\item  cubic Jordan algebras of dimension $r+1$, up to isomorphism;
\item  quadro-quadratic Cremona transformations of $\mathbb P^r$, up to linear equivalence.
\end{itemize}

\smallskip

 In \cite{brunoverra}, see also \cite{Semple} and \cite{Verra},  a classification of quadro-quadric Cremona transformations in $\p^4$ is obtained. This immediately yields also the classification of $\overline X(5,3,3)\subset\p^{11}$ and provides
 an affirmative answer to the above conjecture also for $r=4$. We refrain
from listing this classification and we shall probably come back in a future paper, \cite{PR2}, on the beautiful relations between the above  apparently unrelated  objects,
trying to develop further the classification results and the connections among these areas.}
\end{rem}

\end{document}